\documentclass[11pt]{article}   
 
\usepackage{tikz}
\usetikzlibrary{calc,through,backgrounds,positioning,shapes.misc}
\usetikzlibrary{arrows,decorations.pathmorphing,fit,petri}
 
\usepackage[utf8]{inputenc} 
\usepackage{amsmath,amssymb,amsthm,enumerate,amsthm,flafter}
\usepackage{url}
 
\usepackage{pdflscape}
\usepackage{caption}
\setcaptionwidth{14cm}

\usepackage[colorlinks=false]{hyperref} 

\usepackage{anysize}

\title{Spaces of convex $n$-partitions\footnote{The first author was funded by DFG through the \emph{Berlin Mathematical School}. Research by the second author was supported by the DFG Collaborative Research Center TRR~109 ``Discretization in Geometry and Dynamics''.}}
\author{Emerson León\\
Depto. de Matemáticas\\
Universidad de los Andes\\ 
Bogotá, Colombia\\
\url{emersonleon@gmail.com}
\and
G\"unter M.~Ziegler\\  
Inst.\ Mathematics, FU Berlin\\Arnimallee 2\\14195 Berlin, Germany\\
\url{ziegler@math.fu-berlin.de}}
\date{{\small November 8, 2015}}

\newcommand{\R}{{\ensuremath{\mathbb{R}}}}
\newcommand{\N}{{\ensuremath{\mathbb{N}}}}

\newcommand{\A}{{\ensuremath{\mathcal{A}}}}

\newcommand{\F}{{\ensuremath{\mathcal{F}}}}
\newcommand{\C}{{\ensuremath{\mathcal{C}}}}

\newcommand{\Q}{{\ensuremath{\mathcal{Q}}}}

\newcommand{\g}{{\ensuremath{\boldsymbol{g}}}}
\newcommand{\x}{{\ensuremath{\boldsymbol{x}}}}
\newcommand{\y}{{\ensuremath{\boldsymbol{y}}}}

\newcommand{\s}{{\ensuremath{\boldsymbol{s}}}}
\newcommand{\e}{{\ensuremath{\boldsymbol{e}}}}
\newcommand{\w}{{\ensuremath{\boldsymbol{v}}}}
\newcommand{\p}{{\ensuremath{\boldsymbol{p}}}}

\renewcommand{\v}{{\ensuremath{\boldsymbol{v}}}}
\renewcommand{\a}{{\ensuremath{\boldsymbol{a}}}}
\renewcommand{\b}{{\ensuremath{\boldsymbol{b}}}}
\newcommand{\HH}{{\ensuremath{\mathcal{H}}}}
\newcommand{\D}{{\ensuremath{\mathcal{D}}}}
\renewcommand{\A}{{\ensuremath{\mathcal{A}}}}
\newcommand{\I}{{\ensuremath{\mathcal{I}}}}
\let\oldc\c
\renewcommand{\c}{{\ensuremath{\boldsymbol{c}}}}

\renewcommand{\P}{{\ensuremath{\mathcal{P}}}}

\DeclareMathOperator{\sd}{sd}
\DeclareMathOperator{\vol}{vol}

\DeclareMathOperator{\cone}{cone}

\DeclareMathOperator{\relint}{relint}

\DeclareMathOperator{\reg}{reg}

\definecolor{tinto}{rgb}{0.5,0.1,0}
 
\makeatletter
\newtheoremstyle{break}
   {\topsep}{\topsep}%
   {\itshape}{}%
   {\bfseries}{}%
   {\newline}
   {\thmname{#1}\thmnumber{\@ifnotempty{#1}{ }\@upn{#2}}%
    \thmnote{ {\bfseries(#3)}}}%
\makeatother

\theoremstyle{plain}
\newtheorem{theorem}{Theorem}[section]
\newtheorem{proposition}[theorem]{Proposition}

\newtheorem{conjecture}[theorem]{Conjecture}
\newtheorem{lemma}[theorem]{Lemma}

\theoremstyle{definition}
\newtheorem{definition}[theorem]{Definition}

\newtheorem{example}[theorem]{Example}
 
\begin{document}

\maketitle
 
\begin{abstract}
We construct and study the space $\C(\R^d,n)$ of all partitions of $\R^d$ into $n$ non-empty open convex regions ($n$-partitions).
A representation on the upper hemisphere of an $n$-sphere is used to obtain a metric and thus a topology on
this space. We show that the space of partitions into possibly empty regions  $\C(\R^d,\le n)$
yields a compactification with respect to this metric.
We also describe faces and face lattices, combinatorial types, and adjacency graphs for $n$-partitions, and use
these concepts to show that $\C(\R^d,n)$ is a union of elementary semialgebraic sets.  
\end{abstract}

\section{Introduction}

In 2006, R. Nandakumar and N. Ramana Rao \cite{Nandakumar06} asked whether any convex polygon for any integer $n\ge2$
can be cut into $n$ convex pieces of equal area that also have the same perimeter.
This problem is easily generalized to ask whether any probability measure on~$\R^d$ with a continuous
density function admits a partition of~$\R^d$ into $n$ convex regions that capture equal parts of
the measure and equalize some $d-1$ additional functions on non-empty convex regions.
The generalized Nandakumar--Ramana Rao problem captured a lot of attention 
(see e.g.\ Nandakumar \& Ramana Rao \cite{NandakumarRamanaRao12}, Bárány et al.\ \cite{BaranyBlagojevicSzucs},
Karasev et al.\ \cite{KarasevHubardAronov:equipartion}, and Blagojevi\'c \& Ziegler \cite{BlagojevicZieglerEquipartitions}),
but even the original basic version of the problem is still open in the case when 
$n$ is not a power of a prime.

All approaches to the Nandakumar--Ramana Rao problem and to similar problems start with
constructing suitable configuration spaces, that is, spaces of partitions of $\R^d$ into
$n$ convex regions. In particular, Karasev observed that the classical configuration spaces
$F(\R^d,n)$ of $n$ distinct labelled points in~$\R^d$ can---via optimal transport---be used
to parameterize \emph{regular} $n$-partitions (that is, weighted Voronoi partitions),
while Nandakumar \& Ramana Rao \cite{NandakumarRamanaRao12} used products of spheres to parameterize the
partitions that arise from nested hyperplane $2$-partitions.

Motivated by the Nandakumar--Ramana Rao problem we here consider the set $\C(\R^d,n)$
of \emph{all} partitions of $\R^d$ into  $n$  convex regions, for positive integers $d$ and $n$.  
We describe a natural metric on this set, and thus can treat $\C(\R^d,n)$ as the 
\emph{space of all convex $n$-partitions of $\R^d$}, which in particular is a topological space. 
These spaces for different $n$ and $d$ are our main object of study.
We construct natural compactifications of these spaces. 
One main result is that the spaces $\C(\R^d,n)$ can be described as finite unions of semialgebraic sets. 
Thus, in particular, they have well-defined dimensions. We give two possible ways to decompose spaces 
of $n$-partitions as unions of semialgebraic pieces.  Since all the regions of a partition are polyhedral, 
we obtain one parameterization from the hyperplane description of the regions.
We also define the face structure for each partition and use this to define and distinguish \emph{combinatorial types}. 
(These definitions are far from being straightforward~\ldots)
Realization spaces arise as the spaces of all partitions that share the same combinatorial type. These realization spaces 
also give us semialgebraic pieces that we glue together to obtain the whole space $\C(\R^d,n)$.  
We also discuss the dimensions of these realization spaces.
 
This paper presents main results of the doctoral thesis of the first author \cite{thesis}. 
As far as we know, there is no previous reference of the spaces of convex $n$-partitions in the literature, 
although some similar spaces and particular cases have been studied.  
  

\section{Convex \emph{n}-partitions}\label{firstsection}

We begin here with the definition of convex $n$-partitions. 

\begin{definition}[Convex partitions of $\R^d$, regions, $n$-partitions]\label{npartition}
Let $n$ and $d$ be two positive integers. A \emph{convex partition of\/ $\R^d$} is an ordered list $\P=(P_1,\,P_2,\ldots, P_n)$ of non-empty open convex subsets $P_i\subseteq \R^d$ that are pairwise disjoint, so that the union $\bigcup_{i=1}^n \overline{P_i}$ equals $\R^d$, where $\overline{P_i}$ denotes the closure of $P_i$.  
Each of the sets $P_i$ is called a \emph{region} of~$\P$.
Partitions into $n$ convex regions are also called \emph{$n$-partitions}.
\end{definition}

Since all partitions we are dealing with here are convex, we will often omit this word. The regions of an $n$-partitions are labeled from $1$ to $n$, where the order is important.  
\begin{definition}[Space of convex $n$-partitions]
The set of all convex $n$-partitions of $\R^d$ is denoted by $\C(\R^{d},n)$.
\end{definition}

As any two regions can be separated by a hyperplane
(by the Hahn--Banach Separation Theorem, see Rudin \cite[Theorem 3.4]{rudin1991functional}),
we get that each region in an $n$-partition
can be described as the set of all points that satisfy a finite set of linear inequalities:

\begin{proposition}\label{polyhedral}
Let $\P=(P_1,\,P_2,\ldots, P_n)$ be an $n$-partition of $\R^d$. Then each region ${P_i}$ is the solution set of $n-1$ strict linear inequalities, so it is the interior of a (possibly unbounded) $n$-dimensional polyhedron with at most $n-1$ facets.
\end{proposition}

\subsection{Spherical representation and partitions of $S^d$}
\label{sphsection}

We now introduce convex partitions of the unit $d$-sphere $S^d\subset \R^{d+1}$. Even if we are primarily interested 
in partitions of $\R^d$
(as in Definition \ref{npartition}), 
partitions of the sphere appear naturally and they are in some sense the more fundamental objects, 
which generalize partitions of the Euclidean space $\R^d$, and provide the natural setting for the
definition and discussion of faces (including “faces at infinity”) as well as for the construction of compactifications.

\begin{definition}[Convex partitions of the sphere]\label{sphericalnpartition}
Let $n$ and $d$ be two positive integers. A \emph{convex partition of $S^d$}  is a list $\Q=(Q_1,\,Q_2,\ldots, Q_n)$ of non-empty 
open convex subsets $Q_i\subseteq S^d$ that are pairwise disjoint, so that the union $\bigcup_{i=1}^n \overline{Q_i}$ equals $S^d$.
\end{definition}

A vector $\v=(v_0,\ldots,v_d)\in S^d$, $\|\v\|=1$, lies in the \emph{upper hemisphere} $S^d_+$ if its first coordinate is positive, $v_0>0$. 
Respectively, $\v$ is in the \emph{lower hemisphere} $S^d_-$ if $\v\in S^d$ and $v_0<0$. The \emph{equator} $S^d_0$ 
of $S^d$ is formed by all $\v\in S^d$ with $v_0=0$.
For any $\x\in \R^d$ we construct the point 
\[
    \widehat\x=\frac{1}{\sqrt{1+\|\x\|^2}}\binom{1}{\x}\in\R^{d+1}, 
\] 
that is, the intersection of the ray $r(\x)=\{\lambda \binom{1}{\x}\in\R^{d+1}: 0\le \lambda\in\R\}$ with $S^d$. 
The map $\x\mapsto \widehat\x$ gives a bijection between $\R^d$ and $S^d_+$. 

\begin{definition}[Spherical representation] \label{hemisphere}
The \emph{spherical representation} of an $n$-partition $\P$ of $\R^d$ is the spherical $(n+1)$-partition 
$\widehat \P =(\widehat P_1,\ldots,\,\widehat P_n,\,\widehat P_\infty)$ of $S^d$, 
with regions $\widehat P_i=\{\widehat \x:\x\in P_i\}$ for $i=1,\ldots,n$ 
and an extra region $\widehat P_\infty:=S^d_-$.
Thus for the spherical representation $\widehat \P$ of an $n$-partition $\P\in \C(\R^d,n)$, we denote by $\infty$ the subindex $n+1$.
\end{definition}
 
\begin{proposition}
  The spherical representation $\widehat \P$ of an $n$-partition $\P$ of $\R^d$ is a convex partition of $S^d$ with $n+1$ regions.
\end{proposition}

\begin{example}\label{ex1}
 Figure \ref{example1} shows an $4$-partition $\P$ of $\R^2$ together with an upper view of its spherical representation $\widehat \P$, where we only depict the upper hemisphere $S^2_+$. The face $\widehat P_{\infty}$ corresponds to the side of the sphere hidden to us. This partition includes two parallel lines as the boundary of $P_3$ that in the spherical representation meet at two points ``at infinity'' (on the boundary of $S^2_+$).
\end{example}

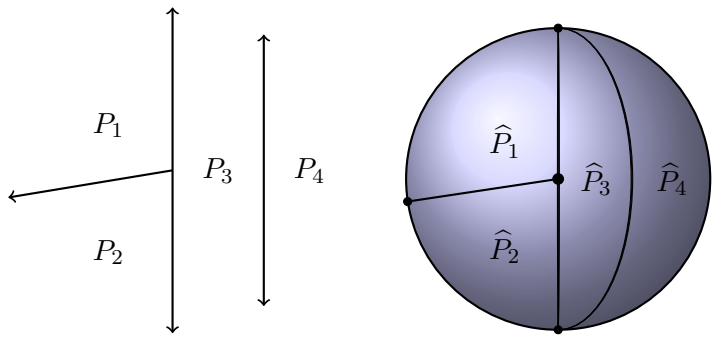
\begin{figure}[htb]\centering
    \begin{tikzpicture}[scale=1.2]
      \draw [thick,<->] (0, -1.8) -- (0, 1.8); 
      \draw [thick,<->] (1, -1.5) -- (1, 1.5); 
      \draw [thick,->] (0, 0) -- (-1.8, -0.3);

      \node at (1.5, 0) {$P_4$};
      \node at (0.5, 0) {$P_3$};
      \node at (-0.7, 0.5) {$P_1$};
      \node at (-0.7, -0.9) {$P_2$};
    \end{tikzpicture}
    \qquad
  \begin{tikzpicture}
    \shade [ball color=blue!20]
    (0,0) circle (2);

    \draw (0, -2) -- (0, 2); 
    \draw (0,0) circle (0.01 and 2);
    \draw (0,2) arc (90:-90:0.015 and 2);

    \draw (0,2) arc (90:-90:0.98 and 2);
    \draw (0,2) arc (90:-90:0.97 and 2);
    \draw (0,2) arc (90:-90:0.96 and 2);%

    \draw [thick] (0, 0) -- (-1.98, -0.3);
    \draw [thick] (0, 0) circle (2);
    \draw [fill] (0, 0) circle (2pt);
    \draw [fill] (0, 2) circle (1.5pt);
    \draw [fill] (0, -2) circle (1.5pt);
    \draw [fill] (-1.98,-0.3) circle (1.5pt);

      \node at (1.5, 0) {$\widehat P_4$};
      \node at (0.5, 0) {$\widehat P_3$};
      \node at (-0.7, 0.5) {$\widehat P_1$};
      \node at (-0.7, -0.9) {$\widehat P_2$};
    \end{tikzpicture}

    \caption{A 4-partition  $\P\in \C(\R^{2},4)$ together with an upper view of its spherical representation.}
\label{example1} 
\end{figure}
  
\subsection{Faces and the face poset}
\label{facedsection}

Since we want to study the behavior of $n$-partitions at infinity as part of the face structure, 
it will be convenient to use for this the spherical representation. First we introduce the faces of spherical partitions. 
Faces will be in correspondence with index sets.
The faces of an $n$-partition $\P$ ordered by inclusion will form the face poset of~$\P$. 
We will see that it is the poset of a finite regular CW complex homeomorphic to a closed ball. 
(As usual, \emph{poset} stands for \emph{partially ordered set}, see e.g.\ Stanley \cite[Chapter~3]{stanley}.) 

\begin{definition}[Index sets and faces of spherical partitions] 
Let $\Q=(Q_1,\ldots,\,Q_{n})$ be a partition of $S^d$.  Let $\overline Q_i$ be the closure of $Q_i$ in $S^d$ and 
$C_i=\cone(\overline Q_{i})$ for $1\le i\le n$.  
For any point $\x$ in $\R^{d+1}$, we define the \emph{index set} $I(\x)$ to be the set of values $i\in \{1,\,2, \ldots, n\}$ 
such that $\x\in C_i$. We define $\I(\Q)$ to be the set of all index sets $I(\x)$ for $\x\in \R^{d+1}$.  

The \emph{faces of a spherical partition} $\Q$ are all sets $F_I\subseteq S^d$ that can be obtained as 
an intersection of the form $F_I=\bigcap_{i\in I} \overline{Q_i}$ for some $I\in \I(\Q)$. 
That is, for each $\x \in \R^{d+1}$ we obtain the spherical face 
\[
        F_{I(\x)}=\bigcap_{i\in I(\x))} \overline{Q_i}\subseteq S^d.
\]
\end{definition} 

\begin{lemma}\label{reversedinclusion}
If $I(\x)\subsetneq I(\x')$ then $F_{I(\x')}\subsetneq F_{I(\x)}$. 
\end{lemma}
 
\begin{proof}
  The inclusion $F_{I(\x')}\subseteq F_{I(\x)}$ is clear since the intersection $F_{I(\x')}= \bigcap_{i\in I(\x')}\overline{Q_i}$ 
  includes all terms involved in computing $F_{I(\x')}$. Also if $I(\x)\neq I(\x')$ then $\x\notin F_{I(\x')}$, 
  since there is at least one $i\in I(\x')- I(\x)$ such that $\x\notin \overline{Q_i}$. Since $\x\in F_{I(\x)}$ 
  we get the strict inclusion $F_{I(\x')}\subsetneq F_{I(\x)}$.
\end{proof}

\begin{definition}[Faces of partitions of $\R^d$, faces at infinity, interior faces, bounded faces]
The \emph{faces of an $n$-partition} $\P$ of $\R^d$ are all the faces of the spherical representation $\widehat \P$, 
with the exception of $F_{\{\infty\}}=\overline{S^d_-}$.
Faces  $F_{I(\x)}$ of~$\P$ with $\infty\in I(\x)$ are called \emph{faces at infinity} of~$\P$. 
All other faces are called \emph{interior faces}. A face is \emph{bounded} if it does not contain any face at infinity.
\end{definition} 

With this definition, faces of an $n$-partition $\P$ of $\R^d$ are \emph{not} subsets of $\R^d$, but they are
subsets of the closure of $S^d_+$. Faces at infinity are precisely the faces of~$\P$ contained in the boundary of $S^d_+$, 
which is the equator $S_0^d$.
For a convex $n$-partition $\P$ we set $\I(\P)=\I(\widehat \P)\setminus \big\{\{\infty\}\big\}$ to be the set of indices of faces of~$\P$. 

Each $n$-partition has only finitely many faces $I(\x)$, as they are subsets of the finite set ${I({\bf 0})}=\{1,\ldots,\,n,\infty\}$ 
(where ${\bf 0}$ represents the origin in $\R^{d+1}$). 
The union of all faces of~$\P$ will be precisely $S^d_+$, since any point $\x\in S^d_+$ is contained in
a face, namely in $F_{I(\x)}$.

\begin{definition}[Face poset]
  The \emph{face poset} of an $n$-partition $\P$ is the set of all faces of~$\P$, partially ordered by inclusion. 
  It is denoted as $\F(\P)$.
\end{definition}
 
\begin{example} 
In Figure \ref{faceposet} we show the face poset of the partition $\P$ on Example \ref{ex1}. Here we denote by $  F_{123}$ the face $  F_{\{1,2,3\}}, $ and similarly for other sets of indices. Notice that $  F_{I({\bf 0})}=  F_{1234\infty}=\emptyset$. To obtain the face poset of $\widehat \P$ we have to add the face $  F_{\infty}$ as another maximal face above all faces at infinity (appearing with dotted lines in the figure).
\begin{figure}[htb]\centering
    \begin{tikzpicture}[inner sep=1pt,outer sep=1pt]
      \node (I0) at ( 0, 0) {$\emptyset$};
      
      \node (123) at (-2.4, 1)  {$F_{123}$};
      \node (012) at (-.8, 1)  {$F_{12\infty}$};
      \node (0134) at (.8, 1)  {$F_{134\infty}$};
      \node (0234) at (2.4, 1)  {$F_{234\infty}$};
      
      \node (12) at (-3, 2.2)  {$F_{12}$};
      \node (13) at (-2, 2.2)  {$F_{13}$};
      \node (23) at (-1, 2.2)  {$F_{23}$};
      \node (34) at (0, 2.2)  {$F_{34}$};
      \node (01) at (1, 2.2)  {$F_{1\infty}$};
      \node (02) at (2, 2.2)  {$F_{2\infty}$};
      \node (04) at (3, 2.2)  {$F_{4\infty}$};
      
      \node (1) at (-2.6, 3.4)  {$F_{1}$};
      \node (2) at (-1.3,  3.4)  {$F_2$};
      \node (3) at (0, 3.4)  {$F_3$};
      \node (4) at (1.3, 3.4)  {$F_4$};
      \node (0) at (2.6, 3.4)  {$F_\infty$};
      
      \draw [-] (I0) -- (123); 
      \draw [-] (I0) -- (012); 
      \draw [-] (I0) -- (0134); 
      \draw [-] (I0) -- (0234);
      
      \draw [-] (123) -- (12); 
      \draw [-] (123) -- (13); 
      \draw [-] (123) -- (23); 
      \draw [-] (012) -- (01); 
      \draw [-] (012) -- (02); 
      \draw [-] (012) -- (12); 
      \draw [-] (0134) -- (01); 
      \draw [-] (0134) -- (04); 
      \draw [-] (0134) -- (34); 
      \draw [-] (0134) -- (13); 
      \draw [-] (0234) -- (02); 
      \draw [-] (0234) -- (04); 
      \draw [-] (0234) -- (34); 
      \draw [-] (0234) -- (23); 
      \draw [-] (01) -- (1); 
      \draw [-] (12) -- (1); 
      \draw [-] (13) -- (1); 
      \draw [-] (02) -- (2); 
      \draw [-] (12) -- (2); 
      \draw [-] (23) -- (2); 
      \draw [-] (13) -- (3); 
      \draw [-] (23) -- (3); 
      \draw [-] (34) -- (3); 
      \draw [-] (04) -- (4); 
      \draw [-] (34) -- (4); 
      \draw [dotted] (0) -- (01); 
      \draw [dotted] (0) -- (02); 
      \draw [dotted] (0) -- (04); 
      
    \end{tikzpicture}
    \caption{Face poset of the partition $\P$ on Example \ref{ex1}.}
    \label{faceposet} 
  \end{figure}
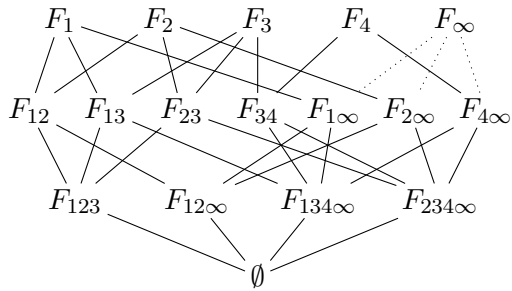
\end{example}
 
Let us provide a brief overview of the most relevant notation we have introduced so far related with an $n$-partition $\P$:
\begin{itemize}
\item $\P=(P_1,\,P_2,\ldots, P_n)$ denotes an $n$-partition of $\R^d$, $\P\in \C(\R^{d},n)$. 
\item $\widehat \P=(\widehat P_1,\ldots, \widehat P_n,\,\widehat P_\infty)$ is the spherical representation of~$\P$, a partition of~$S^d$ into $n+1$ regions.
\item $\I(\P)=\I(\widehat \P)\setminus \big\{\{\infty\}\big\}$ is the set of indices of faces of~$\P$. 
\item $F_{I}\subset S^d$ are the faces of~$\P$, for $I=I(\x)\in \I(\P)$ and $\x\in\R^{d+1}$.  
\item $C_I=\cone(F_I)$ are the corresponding cones in $\R^{d+1}$. 
\end{itemize}
Faces of~$\P$ of dimension $k$ are also known as $k$-faces.
The $0$-faces of~$\P$ are called the \emph{vertices} and the $1$-faces are called  \emph{edges}, but only in case they are contractible. As an example, a partition of $\R^2$ given by parallel lines has a $0$-face that is not vertex and is the union of two points at infinity. We introduce now a set of partitions where such strange effects do not occur.

\begin{definition}[Essential partitions]
  An $n$-partition $\P$ is \emph{essential} if $F_{I({\bf 0})}(\P)=\emptyset$. 
\end{definition}
Since all $I\in \I(\P)$ are contained in $I({\bf 0})=\{1,\ldots,\,n,\,\infty\}$, the face $F_{I({\bf 0})}$ is the minimal face of the partition.  It is easy to check that an $n$-partition $\P$ is essential if and only if it has a bounded face, and moreover, it is essential if and only if it has an interior vertex.

\begin{definition}[Subfaces]\label{def:subfaces}
The \emph{subfaces} of a face $F_I$ of an $n$-partition $\P$ are the faces of $F_I$ considered as a convex spherical polyhedron, i.\ e.\ the faces of the cone $C_I$ intersected with $S^d$ for $I\in\I(\P)$. Subfaces of a face of~$\P$ are also called subfaces of~$\P$. 
\end{definition}

Subfaces of dimension $k$ are denoted as $k$-subfaces. It is easy to verify that
each subface is a union of faces; see \cite[Lemma 3.23]{thesis}. 

\begin{example}\label{example:subfaces}
For the partition in Figure \ref{example1}, the region $\widehat P_3$ is bounded by two subfaces of dimension~$1$. 
One of these subfaces is the face $F_{34}$ of the partition, while the other one is the union of the faces $F_{13}$ and $F_{23}$.
\end{example}

\begin{theorem}\label{cwcomplex}
If $\P$ is an essential $n$-partition of $\R^d$, then the faces of~$\P$ form a regular CW complex homeomorphic to $\overline{S^d_+}$. 
\end{theorem}
  
If the partition is not essential, one can see it as a partition of a lower-dimensional subspace.
  
\begin{proposition}
 The order complex of the face poset of an $n$-partition $\P$ is homeomorphic to a ball of dimension $d-k$, where $k=\dim F_{I({\bf 0})}$.
\end{proposition}
 
For detailed proofs of these results see \cite{thesis}.
To obtain a CW-complex homeomorphic to a $d$-ball, we could instead make a refinement of the faces of any non-essential $n$-partition. One 
way to construct such refinement will be presented in Definition \ref{cwrefinement}, 
where node systems are introduced.


\section{Metric structure, topology and compactification} \label{metricsection}

Understanding the space $\C(\R^d,n)$ of all convex $n$-partitions of $\R^d$ is the main goal of this work. 
In this section we investigate its basic structure. 
First we define a metric on $\C(\R^d,n)$, which induces a topology, 
and introduce a natural compactification, the space $\C(\R^{d},\le\! n)$. 

For non-empty compact convex sets there are two standard ways to measure the distance between them. 
The \emph{Hausdorff distance} between two compact convex sets $A,\,B \subset \R^d$ is defined as  
\begin{equation}
\delta(A,B) = \max \big( \max_{a \in A} \min_{b \in B} \|a-b\|, \max_{b\in B}\min_{a\in A} \|a-b\| \big)
\end{equation}
and the \emph{symmetric difference distance} 
\begin{equation}
\theta(A,B) = \vol_d(A\vartriangle B),
\end{equation}
where $A\vartriangle B$ denotes the symmetric difference of sets $A$ and $B$.
Both of these metrics induce the same topology (see \cite{GruberKenderov}). 
These metrics cannot be used directly for unbounded regions, since then the distances would be typically infinite. 
To remedy this, instead of the usual volume $\vol_d$ on $\R^d$ we use a continuous measure $\mu$ that is finite, 
i.e.\ such that $\mu(\R^d)<\infty$. A measure is \emph{positive} if it is supported on the whole space $\R^d$. 
Throughout our discussion the measures we consider are positive, continuous and finite.

One natural choice for the measure that can be used for any measurable set $P\subseteq \R^d$ is the standard $d$-volume $\mu(P)=\vol_d(\widehat P)$ of the projection to the sphere;
this volume is bounded by $\vol_d(S^d_+)=\frac12\vol_d(S^d)$. 
With this measure $\mu$, we can fix a metric on $\C(\R^{d},n)$ as follows.

\begin{definition}\label{metricdmu}
Given two $n$-partitions $\P=(P_1, \ldots, P_n)$ and $\P'=(P'_1, \ldots, P'_n)$ of $\R^d$,  
their \emph{distance} $d_\mu(\P,\P')$ is the sum of the measures of the symmetric differences  
of the corresponding regions,  
\begin{equation*}
d_\mu(\P,\P') = \sum_{i=1}^n\mu(P_i\vartriangle P'_i).
\end{equation*}
\end{definition}

This distance $d_\mu$ is a metric and endows $\C(\R^d,n)$ with the topology that we use for our study. 
There is a natural compactification for $\C(\R^{d},n)$ that is obtained by considering generalized $n$-partitions 
that are allowed to have empty regions.

\begin{definition}[Non-proper and proper $n$-partitions]\label{nonproperpartitions}%
Let $n$ and $d$ be two positive integers. A \emph{non-proper $n$-partition of $\R^d$} is a list $\P=(P_1,\,P_2,\ldots, P_n)$ of $n$ open convex subsets $P_i\subseteq \R^d$ that are pairwise disjoint, so that the union $\bigcup_{i=1}^n \overline{P_i}$ equals $\R^d$, where now the $P_i$ are allowed to be empty and at least one of the $P_i$ is empty. 
The convex $n$-partitions as introduced in Definition~\ref{npartition} are called \emph{proper} in this context.  
We denote by $\C(\R^{d},\le \!n)$ the set of all proper or non-proper $n$-partitions.
\end{definition}

Thus $\C(\R^{d},n)$ is the subset of proper partitions in $\C(\R^{d},\le\!n)$.  
A non-proper partition can also be seen as a $k$-partition with $k<n$, whose regions have distinct labels
in the range from $1$ to $n$, while labels that are not used correspond to empty regions.  

Most of the results and definitions we have introduced up to now can be extended to non-proper partitions. 
The distance $d_{\mu}$ can be extended to $\C(\R^{d},\le \!n)$, so that it is also a metric and topological space. 
Non-proper partitions also have polyhedral regions as claimed by Theorem \ref{polyhedral} (now possibly empty). 
We can also talk about non-proper partitions of a $d$-sphere, spherical representation of non-proper partitions 
and face structure, where now the labels of the faces $I(\x)$ are contained in ${I({\bf 0})}=\{i:C_i\neq \emptyset\}$, 
the set of labels of all non-empty regions. For a region $P_i=\emptyset$, we define $C_i$ to be empty as well, 
so that we don't get new faces by adding extra empty regions. As before, a non-proper $n$-partition is \emph{essential} 
if $F_{I({\bf 0})}=\emptyset$.  

\begin{theorem}\label{compactness}
The space $\C(\R^{d},\le \!n)$ is compact.
\end{theorem}

\begin{proof}
For the proof we introduce additional spaces that will also be important for the discussion of the semialgebraic
structure in the next section. The first one is $(S^d)^{\binom n 2}$, a compact subset of $\R^{(d+1)\times\binom n2}$. 
Each of the points $\c\in (S^d)^{\binom n 2}$ is represented by $\binom n 2$ unit vectors $\c_{ij}\in S^d$ 
for $1\le i<j\le n$. Each point $\c$ can be identified with a central oriented hyperplane arrangement $\A_\c$ in $\R^{d+1}$,
with $\binom{n}{2}$ hyperplanes $H_{i j}$. 
Each hyperplane $H_{ij}\in \A_\c$ is given by the linear equation $\c_{ij}\cdot\x=0$ and comes with an orientation
given by the vector $\c_{ij}\in \R^{d+1}$.
To keep the symmetry of the notation, $H_{ji}$ denotes the same hyperplane $H_{ij}$ 
with the opposite orientation, whose normal vector is $\c_{ji}=-\c_{ij}$. 

Now let $\D(\R^d, \le n)$ be the set of $n$ labeled, disjoint, possibly empty, 
open polyhedral subsets $(Q_1, \ldots, Q_n)$ of $\R^d$. 
We fix the topological structure of $\D(\R^d, \le n)$ in the same way as we did for $\C(\R^d, \le n)$ 
using the metric structure from Definition \ref{metricdmu}. For this, we take the metric on $\D(\R^d, \le n)$ 
where the distance of two lists $(Q_1, \ldots, Q_n)$ and $(Q'_1, \ldots, Q'_n)$ in $\D(\R^d, \le n)$ is given by 
\[ 
                \sum_{i=1}^n\vol _d(\widehat Q_i\vartriangle \widehat Q'_i),
\]
that is, the sum of the measures of the symmetric differences of the projections to $S^d$ of the pairs of 
corresponding polyhedra in both lists. In this way $\C(\R^d, \le n)$ is a subspace of $\D(\R^d, \le n)$, 
with the corresponding subspace topology.

Equivalently, $\D(\R^d, \le n)$ can be considered the space of $n$ labeled, disjoint, possibly empty, 
open spherical polyhedral subsets of $S^d_+$ (the upper hemisphere), since we can map each polyhedron 
$Q_i\subseteq \R^d$ to the spherical polyhedron $\widehat Q_i$. The space of partitions $\C(\R^d, \le n)$ 
can be considered as the subspace of those lists $(Q_1, \ldots, Q_n)\in \D(\R^d, \le n)$ for which the union 
of the closures of the $Q_i$ is the whole $\R^d$.

We define a map $\pi: (S^d)^{\binom n 2} \rightarrow \D(\R^d, \le n)$ obtained by taking for each $\c\in (S^d)^{\binom n 2}$ the polyhedra
\[
	Q_i = \{\x\in\R^{d}: \c_{ij}\cdot \tbinom{1}{\x}<0\text{ for }1\le j\le n, j\neq i\},   
\] 
for $1\le i\le n$ where $\binom{1}{\x}\in \R^{d+1}$ is the vector obtained by adding to $\x$ a first coordinate equal to~$1$.
In other words, each $Q_i$ is determined by intersecting the halfspaces $\c_{ij}\cdot \binom{1}{\x}<0$ determined by the affine hyperplanes $H_{ij}\in \A$ in $\R^d$, where the orientation of the $\c_{ij}$ indicates the side of $H_{ij}$ that must be taken. We recall the convention that $\c_{ji}=-\c_{ij}$, which implies that all $Q_i$ are disjoint. The polyhedral sets $Q_i$ might be empty. 
 
The map  $\pi: (S^d)^{\binom n 2} \rightarrow \D(\R^d, \le n)$ is continuous: If we move the hyperplanes a small amount, the polyhedra projected to the sphere also change slightly and the sum of the $d$-volume of the symmetric differences must be small. 
 
With this we can now complete the proof of Theorem \ref{compactness}. Since the space $(S^d)^{\binom n 2}$ is compact, the image of the continuous map $\pi$ is also a compact space (see e.g.\ \cite[Theorem 26.5]{Munkres1}).
On this image we have a continuous function $f$ to $\R$, given by 
$f(Q_1,\ldots,Q_n)=\sum_{i=1}^n\vol _d(\widehat Q_i)$.

This is a continuous function, so the preimage of the
maximal value, namely the $d$-volume of $S^d_+$,
is a closed subset of a compact space, so it is compact
as well.  This preimage is denoted by $\HH(R^d, \le n)$ (as explained later in Definition \ref{hspace}). We conclude that $\C(R^d, \le n)$ is compact, as it is the image under $\pi$ of the compact space $\HH(R^d, \le n)$.
\end{proof} 

We cannot claim that $\C(\R^{d},n)$ is also compact, since the limit of a sequence of proper partitions might have empty regions. 
On the other hand, any non-proper partition can be obtained as a limit of proper partitions.
To check this, take a non-proper partition and subdivide one of its regions into one big and some small convex pieces,
to get a proper $n$-partition out of it. If the measure of the small pieces goes to zero, 
in the limit we end up at the non-proper partition we started with. 
Therefore we can think of $\C(\R^{d},\le\! n)$ as a compactification of $\C(\R^{d},n)$. 
 
\section{Semialgebraic structure}\label{semialgsection}
 
A subset of $\R^m$ is semialgebraic if it can be described as a finite union of solution sets of systems given by finitely many polynomial equations and strict inequalities on the coordinates of $\R^m$. In this section we prove that each of the spaces $\C(\R^{d},n)$ and $\C(\R^{d},\le\! n)$ is a union of finitely many pieces that can be parameterized by semialgebraic sets. 
 
We refer to Bochnak, Coste \& Roy \cite{bochnakCosteRoy} and Basu, Pollack \& Roy \cite{basu} as general references on semialgebraic sets.
We will use here some basic results about semialgebraic sets, such as the fact that finite 
unions and intersections of semialgebraic sets are semialgebraic, 
and the fact that the complements of semialgebraic sets are again semialgebraic.
Most notably, we will use the Tarski--Seidenberg Theorem, which says that semialgebraic sets are closed under projections.

\begin{theorem}[Tarski--Seidenberg {(see \cite[Theorem 2.2.1]{bochnakCosteRoy})}]\label{tarskiseidenberg} 
        If $X \subset \mathbb{R}^n \times \mathbb{R}^m$ is a semi\-algebraic set, 
        and if $p$ is the projection onto the first $n$ coordinates, then $p(X)\subseteq\mathbb{R}^n$ is also semialgebraic.
\end{theorem}

We will also use some of the notation introduced in the proof of Theorem \ref{compactness}, 
such as the map $\pi: (S^d)^{\binom n 2} \rightarrow \D(\R^d, \le n)$. 
Note that the space $(S^d)^{\binom n 2}\subset\mathbb{R}^{(d+1)\binom n 2}$ is semialgebraic.

\subsection{Hyperplane description}
 
\begin{definition}[Hyperplane arrangement carrying a partition]
    Let $\P$ be an $n$-partition of~$\R^d$.
An oriented hyperplane arrangement $\A_\c$ for $\c\in (S^d)^{\binom n 2}$  \emph{carries the partition} $\P$ if $\pi(\c)=\P$.
\end{definition}

In other words, an oriented hyperplane arrangement $\A_\c$ for $\c\in (S^d)^{\binom n 2}$  carries the partition $\P=(P_1,\ldots,P_n)$ if the regions $\widehat P_i$ and $\widehat P_j$ are separated by the hyperplane $H_{ij}$, so that $\c_{ij}\cdot\x<0$ for $\x \in \widehat P_i$ and $\c_{ij}\cdot\x>0$ for $\x \in \widehat P_j$.

Following the proof of Proposition \ref{polyhedral}, we can see that for each $n$-partition $\P\in \C(\R^{d},n)$ there is
at least one hyperplane arrangement $\A$ such that it carries it. This hyperplane arrangement is usually not unique.
Similarly, each non-proper partition $\P\in \C(\R^{d},\le\!n)$ is carried by a hyperplane arrangement: 
If a region $P_i$ is empty, any hyperplane $H_{ij}$ that doesn't intersect $P_j$ is good enough to separate these two regions. 
In the case $P_i=\R^d$ we can still take $\c_{ij}=(1,0,\ldots,0)$.  

\begin{example}
In Figure \ref{hyperplanearrangement}(left), we show a partition of $\R^2$ into four regions, 
but to make the example more interesting 
we will consider it as a non-proper partition in $\C(\R^{2},\le\! 5)$, by taking an extra empty region $P_5=\emptyset$.
\begin{figure}[h]\centering
\begin{tikzpicture}

 \coordinate  (A1) at (-2,-1); 
  \coordinate  (A2) at (-2,2); 
  \coordinate  (A3) at (2,2);  
  \coordinate  (A4) at (2,-1); 
  \coordinate  (A5) at (-0.92,0.2);
  \coordinate  (A6) at (0.92,0.5);

   \draw [thick] (A1)--(A5); 
    \draw [thick] (A2)--(A5);
    \draw [thick] (A3)--(A6);
    \draw [thick] (A4)--(A6);
    \draw [thick] (A5)--(A6);
\node [] (r1) at ($0.35*(A1)+0.35*(A2)+0.3*(A5)$) {$P_1$}; 
\node [] (r2) at ($0.2*(A2)+0.2*(A3)+0.3*(A5)+0.3*(A6)$) {$P_2$};
\node [] (r3) at ($0.35*(A3)+0.35*(A4)+0.3*(A6)$) {$P_3$};
\node [] (r4) at ($0.2*(A4)+0.2*(A1)+0.3*(A5)+0.3*(A6)$) {$P_4$};

\node [] (r5) at ($0.8*(A1)+0.2*(A4)-(0,0.5)$) {$P_5=\emptyset$};

\end{tikzpicture}
 \qquad 
\begin{tikzpicture}

 \coordinate  (A1) at (-2,-1); 
  \coordinate  (A2) at (-2,2); 
  \coordinate  (A3) at (2,2);  
  \coordinate  (A4) at (2,-1); 
  \coordinate  (A5) at (-0.92,0.2);
  \coordinate  (A6) at (0.92,0.5);

  \node [above,opacity=0.5] (b1) at ($0.65*(A2)+0.54*(A3)$) {$H_{15}$}; 
  \coordinate  (b2) at ($0.78*(A1)+0.35*(A4)$);
    \draw [opacity=0.5] (b1)--(b2);

  \node [above,opacity=0.5] (b1) at ($0.41*(A2)+0.86*(A1)$) {$H_{25}$}; 
  \coordinate  (b2) at ($0.33*(A3)+0.95*(A4)$);
    \draw [opacity=0.5] (b1)--(b2);

  \node [above,opacity=0.5] (b1) at ($0.22*(A2)+0.97*(A3)$) {$H_{35}$}; 
  \coordinate  (b2) at ($0.6*(A1)+0.63*(A4)$);
    \draw [opacity=0.5] (b1)--(b2);

  \node [above,opacity=0.5] (b1) at ($0.41*(A1)+0.86*(A2)$) {$H_{45}$}; 
  \coordinate  (b2) at ($0.78*(A3)+0.43*(A4)$);
    \draw [opacity=0.5] (b1)--(b2);

   \draw [thick] (A1)--(A5); 
    \draw [thick] (A2)--(A5);
    \draw [thick] (A3)--(A6);
    \draw [thick] (A4)--(A6);
    \draw [thick] (A5)--(A6);
\node [] (r1) at ($0.35*(A1)+0.35*(A2)+0.3*(A5)$) {$P_1$}; 
\node [] (r2) at ($0.2*(A2)+0.2*(A3)+0.3*(A5)+0.3*(A6)$) {$P_2$};
\node [] (r3) at ($0.35*(A3)+0.35*(A4)+0.3*(A6)$) {$P_3$};
\node [] (r4) at ($0.2*(A4)+0.2*(A1)+0.3*(A5)+0.3*(A6)$) {$P_4$};

\node [opacity=0.5] (r5) at ($0.8*(A1)+0.2*(A4)-(0,0.5)$) {$P_5=\emptyset$};

    \draw [dashed] ($2.7*(A5)-1.7*(A1)$)--(A5); 
    \draw [dashed] ($1.8*(A5)-0.8*(A2)$)--(A5);
    \draw [dashed] ($2.1*(A6)-1.1*(A3)$)--(A6);
    \draw [dashed] ($2.3*(A6)-1.3*(A4)$)--(A6);
    \draw [dashed] ($2*(A5)-(A6)$)--(A5);
    \draw [dashed] ($2*(A6)-(A5)$)--(A6);
  \node [above] (b1) at ($0.47*(A2)+0.72*(A3)$) {$H_{13}$}; 
  \coordinate  (b2) at ($0.33*(A1)+0.83*(A4)$);
    \draw [very thick,dashed] (b1)--(b2);
\end{tikzpicture}
 
\caption{Non-proper partition $\P$ in $\C(\R^{2},\le\! 5)$ together with a possible hyperplane arrangement carrying it.}\label{hyperplanearrangement}
\end{figure}
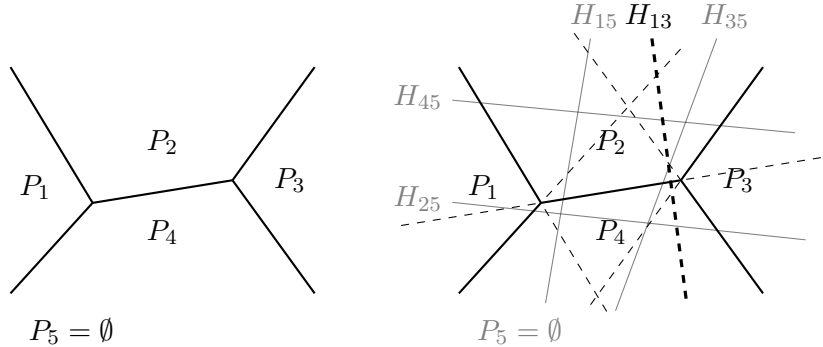

In Figure \ref{hyperplanearrangement}(right) we show an affine picture of a hyperplane arrangement carrying $\P$. 
For adjacent regions $P_i$ and $P_j$, with $\{i,j\}\in A(\P)$, 
there is only one possible hyperplane $H^{\it aff}_{ij}$ that separates them, 
namely the affine span of the points on the intersection of the boundaries.
The extension of these hyperplanes appears on the figure as dashed lines.
For all other hyperplanes there is some freedom to choose them. In the figure, 
there is a label that appears next to each of them. For the hyperplanes involving 
the region $P_5$, it is only necessary that the other region lies entirely on one side of the hyperplane.

The unit vector $\c_{ij}$ is uniquely determined by the hyperplane $H_{ij}$ and the requirement 
that $P_i$ and $P_j$ lie on the correct sides of $H_{ij}$,
unless $P_i=P_j=\emptyset$.  
We remind the reader that an affine hyperplane $H^{\it aff}$ given by the points $\x\in \R^d$ that 
satisfy an equation of the form $\a\cdot\x=b$ for $\a\in \R^d$ and $b\in \R$ is represented projectively 
by its corresponding vector $\c=(-b,a_1,\ldots,a_d)\in \R^{d+1}$ or by the vector $(b,-a_1,\ldots,-a_d)$ 
in case that the opposite orientation is required. This vector can be normalized later. 
\end{example}
  
\begin{definition}[Regions of a hyperplane arrangement] 
 Let $\A_\c$ for $\c\in (S^d)^{\binom n 2}$ be a hyperplane arrangement with hyperplanes  
 $H_{ij}=\{\x\in\R^{d+1}:\c_{ij}\cdot\x=0\}$. Let $\s\in \{+1,-1\}^{\binom{n}{2}}$ be a sign vector 
 with coordinates $s_{ij}\in \{+1,-1\}$ for $i<j$.
 A \emph{region} of the affine hyperplane arrangement $\A_\c^{\it aff}$ is a subset of the form
 \[
 R_\s = \{\x\in\R^{d}:     s_{ij}\c_{ij}\cdot \tbinom{1}{\x} <0\text{ for all }i<j\},
 \] 
 where  $\binom{1}{\x}\in \R^{d+1}$ is the vector obtained by adding a first coordinate equals one to $\x$.
\end{definition}

Thus $R_\s$ is determined by intersecting the halfspaces $s_{ij}\c_{ij}\cdot \binom{1}{\x}<0$ 
determined by the affine hyperplanes $H_{ij}$ of~$\A_\c^{\it aff}$ in $\R^d$, where each coordinate $s_{ij}$ 
indicates the side of $H_{ij}$ that contains $R_\s$, for $1\le i<j\le n$. Regions $R_\s$ 
of a hyperplane arrangement may be empty. 

We symmetrize the notation by setting $s_{ji}=-s_{ij}$ for $i>j$. 
To avoid confusion with the regions of partitions, we always talk about regions $R_\s$  
when we refer to a region of hyperplane arrangements.  
Also, regions $R_\s$ of $\A_\c$  simply denote regions of $\A_\c^{\it aff}$.
 
The \emph{complete graph} $K_n$ is the graph with vertex set $\{1,\ldots,n\}$ and with an edge 
between each pair of vertices. It has $\binom n2$ edges. An \emph{orientation} of $K_n$ is 
obtained by taking the graph $K_n$ and fixing a direction to each edge $e$ of $K_n$, 
by choosing which of the vertices of $e$ is the \emph{tail} and which is the \emph{head}.  
A graph with all its edges oriented is also known as a \emph{directed graph}. 
Each sign vector $\s\in \{+1,-1\}^{\binom{n}{2}}$ generates an orientation $G_\s$ of the complete graph~$K_n$, 
where the edge $ij$ is directed from $i$ to $j$ if $s_{ij}=+1$, and from $j$ to $i$ otherwise, for $1\le i<j\le n$. 
A \emph{source} of $G_\s$ is a vertex $v$ of $K_n$ that is not the head of any of the edges involving $v$ in $G_\s$.
Since the graph $K_n$ is complete, there can be at most one source in the directed graph $G_\s$.

\begin{lemma}\label{carrycondition}
An oriented hyperplane arrangement $\A_\c$ for  $\c\in (S^d)^{\binom n 2}$ carries a (possibly non-proper) 
$n$-partition $\P$ if and only if  for all non-empty regions $R_\s$ of $\A_\c$ the oriented complete graph $G_\s$ has a source.
The partition is proper if and only if for each $i\in\{1,\ldots,\,n\},$ there is at least one non-empty region 
whose source is the vertex $i$.
\end{lemma}

\begin{proof}
If $\P$ is an $n$-partition carried by $\A_\c$, all non-empty regions $R_\s$ of $\A_\c$ must be contained 
in some fixed region $P_i$  of~$\P$.  
If $R_\s\subseteq P_i$, then we have that $s_{ij}=+1$ for all $j\neq i$, so $i$ is a source in $G_\s$. 

On the other hand, if the directed graphs of all non-empty regions $R_\s$ have a source, 
then we obtain an $n$-partition by taking 
\[
        P_i =\bigcap_{j\neq i}\{\x\in \R^{d}:\c_{ij}\cdot\binom{1}{\x}\le 0\}.
\]  
The regions $P_i$ are clearly disjoint, and their union cover all regions $R_\s$ of the hyperplane arrangement, 
since $R_\s\subseteq P_i$ whenever $i$ is the unique source of $G_\s$. Therefore the union of the closure 
of the regions must be the whole $\R^d$. The regions $P_i$ as defined might still be empty, but if there 
is a non-empty region $R_\s$ in $\A_\c$ with $G_\s$ having as source the vertex $i$ for each $i$, then 
$R_\s\subseteq P_i$ is non-empty and the partition is proper.
\end{proof}

\begin{lemma}\label{rs}
For any $\s\in\{+1,-1\}^{\binom{n}{2}}$, the set of points $\c\in (S^d)^{\binom{n}{2}}$ for which
the region $R_\s$ in the hyperplane arrangement $\A_\c$ is empty is semialgebraic. 
The set of points $\c$ such that the region $R_\s$ in the hyperplane arrangement $\A_\c$ 
is non-empty is also semialgebraic.
\end{lemma}

\begin{proof}
The region $R_\s$ is non-empty if and only if there is some $\x\in \R^{d+1}$ such that $s_{ij}\c_{ij}\cdot \x<0$ 
for each pair $i<j$. We can add the coordinates of $\x$ as slack variables and construct 
a semialgebraic set $X$ on the coordinates of $\c_{ij}$ for $1\le i<j\le n$ and of $\x$, 
so that all inequalities $s_{ij}\c_{ij}\cdot \x<0$ are satisfied. The parameterization of 
the set of all hyperplane arrangements $\A_\c$ with $R_\s\neq \emptyset$ can be obtained 
as a projection of $X$ to the coordinates $\c\in (S^d)^{\binom n 2}$ and by Theorem~\ref{tarskiseidenberg}, 
we conclude that the set of arrangements with $R_\s\neq \emptyset$ is semialgebraic. 

Since the complement of a semialgebraic set is semialgebraic, the set of arrangements such that $R_\s= \emptyset$ is semialgebraic. Alternatively, we can use a suitable version of the Farkas Lemma (see \cite[Section 1.4]{Z}) to get another semialgebraic description of this set. 
\end{proof}
 
\begin{definition}[The spaces $\HH(\R^d,\le\! n)$ and $\HH(\R^d,n)$]\label{hspace}
We denote by $\HH(\R^d,\le\! n)$ the space of all $\c \in (S^d)^{\binom n 2}$ such that the hyperplane arrangement $\A_\c$ carries 
a possibly non-proper $n$-partition of $\R^d$. The subset of $\HH(\R^d,\le\! n)$ corresponding to hyperplane arrangements carrying a proper $n$-partition is denoted as $\HH(\R^d,n)$. 
\end{definition}
 
We have the following chain of inclusions.
\[ \HH(\R^d,n) \ \subseteq\ \HH(\R^d,\le\!  n) \ \subseteq\ (S^d)^{\binom n 2} \subseteq\  \R^{(d+1)\times\binom n2}.
\] 

\begin{theorem}\label{thm:semialgrebraic_spaces}
The spaces $\HH(\R^d,\le\! n)$ and $\HH(\R^d,n)$ are semialgebraic sets.
\end{theorem}

\begin{proof} 
By Lemma \ref{carrycondition}, a hyperplane arrangement $\A_\c$ for $\c\in (S^d)^{\binom n 2}$ carries an $n$-partition $\P$ if and only if  for all regions $R_\s$ in $\A_\c$ the oriented graph $G_\s$ have a source. 
Therefore, we need to characterize all hyperplane arrangements $\A_\c$ such that all regions $R_\s$ of $\A_\c$ are empty for all sign vector $\s$ in $S=\{\s\in \{+1,-1\}^{\binom{d}{2}}:G_\s \text{ has no source}\}$. By Lemma~\ref{rs} and the (obvious) fact that finite intersections of semialgebraic sets are semialgebraic, we find that $\HH(\R^{d},\le\! n)$ is a semialgebraic set over the coordinates of $\c_{ij}$ as variables.
 
Also the set $\HH(\R^{d},n)$ of hyperplane arrangements carrying a proper $n$-partition, where at least one region $R_\s$ has source $i$ for each $i\le n$, is  semialgebraic, again by Lemma \ref{rs} and the fact that finite unions and intersections of semialgebraic sets are again semialgebraic.
\end{proof}

From Theorem \ref{thm:semialgrebraic_spaces} we can see that $\HH(\R^d,\le n)$ is the union of all sets of arrangements with an adjacency graph that satisfies the source conditions specified on Lemma \ref{carrycondition}.

\begin{theorem}\label{mainsemialgebraictheorem}
The space $\C(\R^{d},\le\! n)$ is the union of finitely many subspaces 
indexed by adjacency graphs $A(\P)$, which can be parameterized as semialgebraic sets. The same statement is true for the space $\C(\R^{d},n)$.  
\end{theorem}

\begin{proof}
The map $\pi:\HH(\R^{d},\le\! n)\rightarrow \C(\R^{d},\le\! n)$ is a surjective continuous map taking each oriented hyperplane arrangement $\A$ in $\HH(\R^{d},\le\! n)$ to its corresponding partition.
The pieces of $\C(\R^{d},\le \!n)$ are given by the  
partitions in $\C(\R^{d},\le\! n)$ that share the same adjacency graph $A(\P)$, for any given $n$-partition $\P$. Each of these pieces is denoted as $\C_{A(\P)}(\R^{d},n)$ for each partition $\P$, and the inverse image $\pi^{-1}(\C_{A(\P)}(\R^{d},n))$ is denoted as $\HH_{A(\P)}(\R^{d},n)$. 

To see that $\HH_{A(\P)}(\R^{d},n)$ is a semialgebraic set, we take the description of $\HH(\R^{d},n)$ and add extra restrictions to express the fact that certain hyperplanes do not determine any $(d-1)$-face of the partition. These extra restrictions are described in what follows. 

A pair $\{i,j\}$ is in $A(\P)$ for $\P=\pi(\A)$ if and only if there are $\s,\,\s'\in \{+1, -1\}^{\binom{n}{2}}$ with exactly the same entries, except only by the entry $s_{ij}=-s'_{ij}$, with oriented graphs $G_\s$, $G_{\s'}$ having sources $i$ and $j$ respectively, so that the regions $R_\s$, $R_{\s'}$  are non-empty.  

Using Lemma \ref{rs} we find that the subset of arrangements $\A'\in \HH(\R^{d},n)$ with $\{i,j\}\in A(\pi(\A))$ for a given $\A\in \HH(\R^{d},n)$ is semialgebraic, since it is the union over all pairs $\s$, $\s'$  that differ only in the $ij$-coordinate and with respective graphs sources $i$ and $j$ of the subsets of $\HH(\R^{d},n)$ where $R_\s$ and $R_{\s'}$ are non-empty. The complement of those subsets, that represent hyperplane arrangements with $\{i,j\}\notin A(\pi(\A))$ are also semialgebraic.  

Finally $\HH_{A(\P)}(\R^{d},n)$ is the intersection of subsets of $\HH(\R^{d},n)$ where $\{i,j\}\in A(\pi(\A))$ for $\{i,j\}\in A(\P)$ and $\{k,\ell\}\notin A(\pi(\A))$ for $\{k,\ell\}\notin A(\P)$ and thus it is also a semialgebraic set. Since the map $\pi:\HH(\R^{d},\le\! n)\rightarrow \C(\R^{d},\le\! n)$ is a projection obtained by deleting the coordinates $\c_{ij}$ of the hyperplanes $H_{ij}$ for $\{i,j\}\notin A(\P)$, by Theorem \ref{tarskiseidenberg} we conclude that $\C_{A(\P)}(\R^{d},n)$ is a semialgebraic set on the coordinates of the vectors $\c_{ij}$ for $\{i,j\}\in A(\P)$ and $\C(\R^{d},\le\! n)$ is a union of semialgebraic pieces. 

If there are two or more non-empty regions in $\P$, the vertices of $A(\P)$ contained in at least one edge  correspond to the non-empty regions of~$\P$. 
Therefore, we can obtain $\C(\R^{d},n)$ as the union of the semialgebraic pieces of the form $\C_{A(\P)}(\R^{d},n)$ where $A(\P)$ is a connected graph on the vertices from~$1$ to~$n$.
\end{proof}

Only knowing these semialgebraic pieces it is not enough to reconstruct the spaces  $\C(\R^{d},n)$ and $\C(\R^{d},\le\! n)$. 
We also need the topological structure induced by the metric given in Section~\ref{metricsection} to know how to glue the different semialgebraic pieces of the form $\C_{A(\P)}(\R^{d},n)$ in order to obtain the spaces of $n$-partitions $\C(\R^{d},n)$ and $\C(\R^{d},\le\! n)$. 

On each semialgebraic piece $\C_{A(\P)}(\R^{d},n)$ we have a topological structure 
by seeing it as a subset of $\R^{(d+1)\times E}$ given by the parameterization through the $\c_{ij}$, 
where $E$ is the number of edges in $A(\P)$. 
This topological structure is equivalent as the one obtained as a subset of 
$\C(\R^{d},\le\! n)$: To see this, notice that a sequence of partitions $(\P^k)_{k\in \N}$  
in $\C_{A(\P)}(\R^{d},n)$ converges to a partition $\P$ in the $\delta_\mu$-topology if and only if 
each sequence of coordinates $c^k_{ij}$ of the parameterizations of $\P^k$ for $\{i,j\}\in A(\P)$ 
converge to the coordinates of $\c_{ij}$.

\subsection{Node systems and combinatorial types}\label{pointedsection}

Pointed partitions are an important class of partitions, where every face is completely determined 
by its set of vertices.  For general partitions the same doesn't hold and we need to define node systems 
to get similar properties for any $n$-partition (Definition \ref{nodesystem}).

\begin{definition}[Pointed partitions]\label{pointedp}  
        An $n$-partition $\P=(P_1,\dots,P_n)$ of $\R^d$ is \emph{pointed} if for each region $P_i$ 
        the cone $C_i$ is pointed.
\end{definition}

Recall from the comments after Definition \ref{nonproperpartitions} 
that we exceptionally defined $C_i=\emptyset$ in the case a region $P_i=\emptyset$. 
Thus pointed partitions must be proper.
Here we state some simple results about pointed partitions, which are proved in~\cite{thesis}.

\begin{proposition}\label{pointedfacesbyvertices}
If $\P$ is a pointed $n$-partition, then every face $F_I$ of~$\P$ can be obtained as 
the spherical convex hull of all vertices in $F_I$.
\end{proposition}

\begin{lemma}\label{pointedareessential} 
        Pointed $n$-partitions are essential. 
\end{lemma}

The converse of Lemma \ref{pointedareessential} is not true:  Example~\ref{ex1}
shows an example of a partition that is essential but not pointed.

Now we define the node systems of an $n$-partition, in order to get that every face is the spherical convex hull of its corresponding nodes and so that for a pointed partition $\P$ the nodes coincide with the vertices of~$\P$. First we need to introduce the half-linear faces of a partition.

\begin{definition}[Half-linear faces]
A face $F$ of a partition $\P$ is \emph{half-linear} if it is the intersection of $S^d$ 
with a linear subspace of $\R^{d+1}$ and a unique closed halfspace given by a linear inequality. 
The set of half-linear faces of a partition is denoted as $\F^H(\P)$.
\end{definition}

If a face $F$ is half-linear, then it has a unique linear subface $F'$ in its relative boundary. 
The subface $F'$ is the union of some faces of~$\P$, and is the intersection of a linear subspace with $S^d$.
 $F'$ cannot have any boundary since it is topologically a sphere (of dimension $\dim F'=\dim F-1$) 
and is the union of some faces at infinity of~$\P$. Since $\hat P_{\infty}$ is not a face of~$\P$, 
in particular it is not a half-linear face of~$\P$ (but it is a half-linear face 
of the spherical partition $\hat\P$).

The only face $F_I$ of~$\P$ such that its corresponding cone $C_{I}$ is a linear subspace 
is the minimal face $F_{I({\bf 0})}$ (see \cite[Lemma 3.22]{thesis}). This face has no boundary and no subfaces. All faces covering  $F_{I({\bf 0})}$ in the face poset are half-linear. If a partition is essential, all vertices are half-linear faces. 

\begin{example}\label{halflinearexample}
For the $4$-partition $\P$ of Example \ref{ex1}, every vertex is half-linear (there are four of them). 
Besides, there are two more half-linear faces in the figure, namely the faces $F_{34}$ and $F_{4\infty}$. 
For these two $1$-faces, there is a unique linear subface that covers the relative boundary 
and is the union of two vertices of~$\P$.
\end{example}
\begin{example}\label{halflinearexample2}
A $4$-partition $\P'$ of the plane given by four regions separated by three parallel lines 
is non-essential: Its minimal face $F_{I({\bf 0})}(\P')$ consists of 
two antipodal points. Here all $1$-faces are half-linear, since they cover~$F_{I({\bf 0})}(\P')$ and there are no other half-linear faces on this partition.   
\end{example}

\begin{definition}[Node systems, nodes]\label{nodesystem}
Let $\P$ be a partition in $\C(\R^{d},\le\! n)$. 
If $\P$ is essential, a \emph{node system} $N$ of~$\P$ is a set of points $\w_F$, 
one in the relative interior of each half-linear face $F$ of~$\P$.
If the partition $\P$ is non-essential, with $\dim F_{I({\bf 0})}=k\ge0$, then a node system again contains one point $\w_F$ in the relative interior of each half-linear face $F$ of~$\P$, and additionally an ordered sequence of $k+2$ extra points $\w_1,\ldots,\,\w_{k+2}$ on the face $F_{I({\bf 0})}$ such that they positively span the linear subspace $C_{I({\bf 0})}$.  

The points in a node system are referred as \emph{nodes}.
We denote by $N(\P)$ the set of all node systems of~$\P$. 
Note that all vertices of~$\P$ are also nodes in any node system of~$\P$.  
\end{definition}
 
We sometimes write $\w_F(N)=\w_F\in N$, in case it might not be clear which node system we are using.  
If $\P$ is non-essential, the same applies to the nodes $\w_i$ in the minimal face.  

\begin{example}
Here we construct node systems for both partitions of the examples \ref{halflinearexample} and \ref{halflinearexample2}. 
In the first case, every vertex of~$\P$ must be a node. 
We need to include two more nodes $\w_{F_{34}}$ and $\w_{F_{4\infty}}$ in the relative interior 
of the faces $F_{34}$ and $F_{4\infty}$ respectively. We have one degree of freedom to choose each 
of these two nodes. In Figure \ref{nodesystemfig}(left) we depict one possible choice 
for a node system $N$ of~$\P$.  
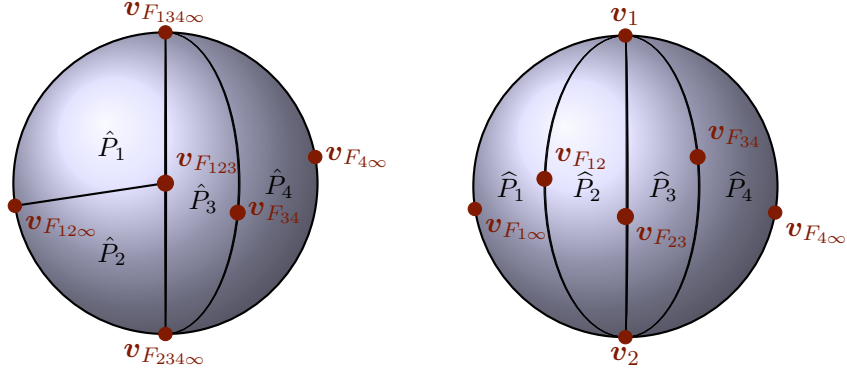
\begin{figure}[htb]\centering
  \begin{tikzpicture}
    \shade [ball color=blue!15]
    (0,0) circle (2);

    \draw (0, -2) -- (0, 2); 
    \draw (0,0) circle (0.01 and 2);
    \draw (0,2) arc (90:-90:0.015 and 2);

    \draw (0,2) arc (90:-90:0.98 and 2);
    \draw (0,2) arc (90:-90:0.97 and 2);
    \draw (0,2) arc (90:-90:0.96 and 2);%

    \draw [thick] (0, 0) -- (-1.98, -0.3); 
    \draw [thick] (0, 0) circle (2);

\draw [fill,tinto] (0, 0) circle (3pt) node[above right] {$\w_{F_{123}}$}; 
\draw [fill,tinto] (0, 2) circle (2.5pt) node[above] {$\w_{F_{134\infty}}$};
\draw [fill,tinto] (0, -2) circle (2.5pt) node[below] {$\w_{F_{234\infty}}$};
\draw [fill,tinto] (-1.98,-0.3) circle (2.5pt) node[below right] {$\w_{F_{12\infty}}$};
\draw [fill,tinto] (10:2) circle (2.5pt) node[right] {$\w_{F_{4\infty}}$};
\draw [fill,tinto] (0.95,-0.39) circle (2.8pt) node[right] {$\w_{F_{34}}$};

      \node at (1.5, 0) {\small $\hat P_4$}; 
      \node at (0.5, -0.2) {\small $\hat P_3$};
      \node at (-0.7, 0.5) {\small $\hat P_1$};
      \node at (-0.7, -0.9) {\small $\hat P_2$};

    \end{tikzpicture}
\qquad
  \begin{tikzpicture} 
    \shade [ball color=blue!15]
    (0,0) circle (2);

    \draw (0,2) arc (90:-90:0.98 and 2);
    \draw (0,2) arc (90:-90:0.97 and 2);
    \draw (0,2) arc (90:-90:0.96 and 2);%

    \draw (0,2) arc (90:270:1.07 and 2);
    \draw (0,2) arc (90:270:1.06 and 2);%
    \draw (0,2) arc (90:270:1.05 and 2);

    \draw (0,2) arc (90:-90:0.01 and 2);
    \draw (0,2) arc (90:-90:0.02 and 2);
    \draw (0,2) arc (90:-90:0.03 and 2);
    \draw (0,2) arc (90:-90:0.035 and 2);

    \draw [thick] (0, 0) circle (2);
    \draw [fill] (0, 2) circle (1.5pt);
    \draw [fill] (0, -2) circle (1.5pt);
    \node at (-1.5, -0) {\small $\widehat P_1$};
    \node at (-0.5, -0) {\small $\widehat P_2$};
    \node at (0.5, -0) {\small $\widehat P_3$};
    \node at (1.5, -0) {\small $\widehat P_4$};
 
    \draw [fill,tinto] (0, -0.4) circle (3pt) node[below right] {$\w_{F_{23}}$}; 
    \draw [fill,tinto] (0, 2) circle (2.5pt) node[above] {$\w_1$};
    \draw [fill,tinto] (0, -2) circle (2.5pt) node[below] {$\w_2$};
    \draw [fill,tinto] (-1.98,-0.3) circle (2.5pt) node[below right] {$\w_{F_{1\infty}}$};
    \draw [fill,tinto] (-10:2) circle (2.5pt) node[below right] {$\w_{F_{4\infty}}$};
    \draw [fill,tinto] (0.95,0.39) circle (2.8pt)  node[above right] {$\w_{F_{34}}$};
    \draw [fill,tinto] (-1.065,0.1) circle (2.8pt) node[above right] {$\w_{F_{12}}$};

  \end{tikzpicture}

    \caption{Node systems for two different $4$-partitions. There is a node in the relative interior of each half-linear face.}
\label{nodesystemfig} 
\end{figure}
 
For the second partition (in Figure \ref{nodesystemfig}(right)), 
we need to have two nodes $\w_1$ and $\w_2$  on the linear face $F_{I({\bf 0})}(\P')$, 
so that they positively span $C_{I({\bf 0})}$. 
There are two possibilities to choose $\w_1$, and $\w_2$ must be the antipodal point $-\w_1$. 
Besides these two nodes, we need five more nodes, one in the relative interior of each 
half-linear face, to get a node system $N'$ of $\P'$.
\end{example}

\begin{proposition}\label{np1}
If $\P$ is an essential $n$-partition, then the set $N(\P)$ of all node systems is a semialgebraic set of dimension $\dim N(\P)=\sum_{F\in \F^H} \dim(F)$.

If $\P$ is non-essential and $k=\dim(F_{I({\bf 0})})$, 
the set $N(\P)$ has dimension 
\[
	\dim N(\P)=k(k+2)+\sum_{F\in \F^H} \dim(F).
\]
\end{proposition} 

The proof of this proposition and a more precise description of the set of all node systems for a given partition can be found in \cite[Proposition 4.29]{thesis}.

\begin{lemma}\label{convexhullofnodes}
If $N$ is a node system of a partition $\P$, then any face $F$ of~$\P$ can be obtained as the spherical convex hull of the set of nodes in $N$ contained in $F$. 
\end{lemma}

\begin{proof}
By induction on the dimension of $F$, first take $F=F_{I({\bf 0})}$ to be the minimal face of~$\P$, with $\dim(F_{I({\bf 0})})=k$. We have $k+2$ nodes in $F_{I({\bf 0})}$ that positively span the $(k+1)$-dimensional linear subspace $C_{I({\bf 0})}$, and therefore its spherical convex hull is equal to $F_{I({\bf 0})}$. If the partition is essential, $F_{I({\bf 0})}=\emptyset$ doesn't contain any node, and its convex hull is also empty.

Now suppose that $\dim(F)=m$ and every face $F'$ of~$\P$ with  $\dim F'<m$ is equal to the convex hull of the nodes contained $F'$. If $F$ is half-linear, we have an extra node $\w_F$ in the interior of $F$, and any other point $\x$ in $F$ is in an interval between $\w_F$ and a point $\x'$ in the boundary of $F$. Since $\w_F$ cannot be antipodal to $\x'$, then $\x$ can be written as a positive combination of $\w_F$ and $\x'$. By the induction hypothesis, $\x'$ is a positive combination of the nodes in the face where it belongs that are also contained in $F$. We use here that subfaces are union of faces. Therefore $\p$ is in the spherical convex hull of the nodes in $F$.

If $F$ is not half-linear, we can find a node $\w$ in the boundary such that its antipodal point 
is not in $F$. Now we can repeat the argument given before, and the result follows.
\end{proof}
  
\begin{definition}[Cell complex from a node system]\label{cwrefinement}
For any node system $N$ of a partition~$\P$, there is a CW complex $\P_N$ such 
that the vertices of this complex are precisely the nodes in $N$, 
and such that each face of~$\P$ is union of faces of~$\P_N$. 
The \emph{complex $\P_N$} is obtained recursively as follows:
\begin{itemize} \itemsep=0pt
\item Include a face $F_S$ in $\P_N$ for every subset $S$ of nodes contained in the minimal face $F_{I({\bf 0})}$, with $|S|\le k+1$, where $F_S$ is the spherical convex hull of $S$ and $k=\dim F_{I({\bf 0})}$. For essential partitions, only the empty set is included in this step.
\item For every half-linear face $F$ of~$\P$ such that the boundary is already covered by faces of $\P_N$, the spherical convex hull of every face $G$ of $\P_N$ contained on the boundary of $F$ together with the  node $\w_F$ is also a face of $\P_N$. (These faces of $\P_N$ are pyramids over the faces on the boundary of $F$.)
\item All other faces of~$\P$ that are not linear or half-linear are also faces of $\P_N$.
\end{itemize} 
\end{definition}

\begin{example}
For the two partitions given in Figure \ref{nodesystemfig}, the cell complex obtained from this construction coincides precisely with what is shown in the picture, where every half-linear $1$-face is subdivided in two segments and every non-pointed region forms a $2$-cell with four nodes and four $1$-faces on the boundary. For a more illustrative  example, consider the $1$-partition of $\R^2$ into one region. This ``partition'' is non-essential, with minimal face $F_{I({\bf 0})}=F_{1\infty}$ of dimension one, equals to the boundary of $\overline S^d_+$ (this face is homeomorphic to $S^1$ and cannot be a cell). There is also one half-linear face $F_1$, that coincides with $\overline S^d_+$. Therefore a node system here would have four nodes, three on the boundary face $F_{I({\bf 0})}$ that positively span the plane containing that face, and one more node $n$ in the interior of $\overline S^d_+$.
The cell complex in this case is obtained by first taking the spherical convex hull of every subset of nodes in the boundary with two or less elements, that form a subdivision of $F_{I({\bf 0})}$ in three edges and three vertices, and then taking the pyramid over all those faces, with apex on the interior node $n$, to obtain a cell decomposition as shown in Figure \ref{1partition}.
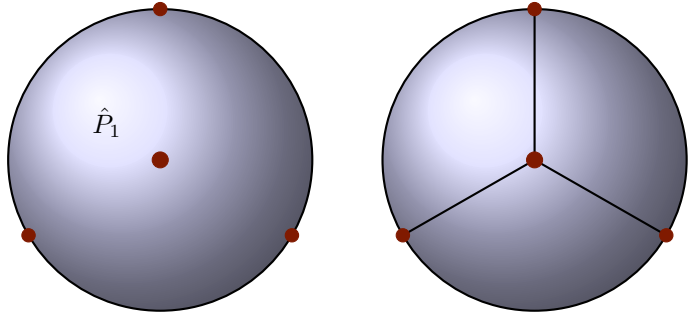
\begin{figure}[htb]\centering
  \begin{tikzpicture}
    \shade [ball color=blue!15]
    (0,0) circle (2);

    \draw [thick] (0, 0) circle (2);

    \draw [fill,tinto] (0, 0) circle (3pt);
    \draw [fill,tinto] (0, 2) circle (2.5pt);
    \draw [fill,tinto] (-30:2) circle (2.5pt);
    \draw [fill,tinto] (210:2) circle (2.5pt);

    \node at (-0.7, 0.5) {\small $\hat P_1$};

  \end{tikzpicture}
\qquad
  \begin{tikzpicture} 
    \shade [ball color=blue!15]
    (0,0) circle (2);

    \draw [thick] (0, 0) circle (2);

    
    \draw [thick] (0, 0) -- (0,2);  
    \draw [thick] (0, 0) -- (-30:2);  
    \draw [thick] (0, 0) -- (210:2);  
    
    \draw [fill,tinto] (0, 0) circle (3pt);
    \draw [fill,tinto] (0, 2) circle (2.5pt);
    \draw [fill,tinto] (-30:2) circle (2.5pt);
    \draw [fill,tinto] (210:2) circle (2.5pt);

  \end{tikzpicture}

    \caption{Node system and cell complex $\P_N$ corresponding to the partition of $\R^2$ with only one region}
\label{1partition} 
  \end{figure}
\end{example}

\begin{lemma}
The complex $\P_N$ is a regular CW complex homeomorphic to a $d$-ball.
\end{lemma}

\begin{proposition}\label{samecomplexforpointed}
For a  pointed partition $\P$, the complex $\P_N$ coincides with the cell complex $\P$ described in Theorem \ref{cwcomplex}. The vertices of $\P_N$ are precisely the vertices of~$\P$.
\end{proposition}

\begin{proof}
For essential partitions all vertices are half-linear faces, and the corresponding node must be precisely at the vertex. If the partition $\P$ is pointed, there are no other half-linear faces, since the cone of a half-linear face $F$ of dimension $\dim F\ge 1$ contains antipodal points on its boundary and therefore is not pointed. Then no other nodes are included, and all faces of $\P_N$ are precisely the faces of~$\P$, so that we end up with the same complex.
\end{proof}

\begin{lemma}
Let $\P$ be a fixed $n$-partition. Then the combinatorial structure of the complex $\P_N$ does not 
depend on the choice of the nodes in $N$, i.\,e.\ for any two node systems $N, N'\in N(\P)$ the face posets of the complexes $\P_N$ and $\P_{N'}$ 
are isomorphic and the complexes are cellularly homeomorphic. 
\end{lemma}

\begin{proof}
The face poset of $\P_N$ can be obtained from the face poset of~$\P$, once we know which are the linear and half-linear faces, independently of the choice of the node system $N$. Following the construction in Definition \ref{cwrefinement}, we can  obtain the face poset of $\P_N$ from the face poset of~$\P$.  
\end{proof}

\begin{definition}[Flags, node frames, node bases, flats]
  Let $\P$ be an $n$-partition, together with a node system $N$. A \emph{flag} of faces of $\P_N$ is a list of faces $G_0\subset \cdots\subset G_d$ completely ordered by containment.
  A \emph{node frame} of $N$ is a list $(\v_0,\ldots,\,\v_d)$ of $d+1$ different nodes in $N$ such that the nodes $\v_0,\ldots,\v_k$ are contained on a $k$-face $G_k$ of $\P_N$ for all $k\le d$ and the faces $G_0\subset \cdots\subset G_d$ form a flag. A \emph{node basis} is a node frame whose vectors are linearly independent and a \emph{flat} is a node frame whose vectors are linearly dependent.
\end{definition}

As the vertices of $\P_N$ are precisely the nodes in $N$ and the face poset of $\P_N$ is 
the same for any node system $N$, for any node frame $(\v_0,\ldots,\,\v_d)$ and 
any other node system $N'$ of~$\P$, the corresponding list of nodes 
$(v_0(N'),\ldots,\,v_d(N'))$ is a node frame of $N'$.  
Also any two partitions $\P$ and $\P'$ with the same face poset and the 
same corresponding half-linear faces have a bijection between node frames, 
as node frames can be read completely from the combinatorial structure of $\P_N$. 

\begin{lemma}\label{signsofflags}
Let $G_0\subset \cdots\subset G_d$ be a complete flag of faces in the complex $\P_N$. 
Then for any list $\x_0,\ldots,\x_k$ of linearly independent vectors in $S^d$ such that 
$\x_i\in G_i$ for all $0\le i\le d$, the sign of the determinant $\det(\x_0,\ldots,\,\x_d)$ 
is given uniquely by the flag $G_0\subset \cdots\subset G_d$.
\end{lemma}

\begin{proof}
Let  $\b_0,\ldots,\b_d$ be the basis of $\R^{d+1}$ where $\b_0\in G_0$ and every  
$\b_i$ for $0<i\le d$ is the vector in the linear space spanned by the face $G_i$ 
orthogonal to the subspace spanned by $G_{i-1}$, such that any point $\x\in G_i$ 
satisfy the inequality $\b_i\cdot \x\ge 0$. 
This basis is uniquely defined by the flag $G_0\subset \cdots\subset G_d$.

Then the vectors $(\b_0,\ldots,\,\b_i)$ span the same linear subspace as the face $G_i$. In terms of this basis, the list of vectors $(\x_0,\ldots,\,\x_d)$ is represented by an upper triangular matrix 
\[\begin{pmatrix}
1 & a_{01} & \cdots & a_{0d}\\ 
0 & a_{11} & \cdots & a_{1d}\\ 
\vdots & \vdots & \ddots & \vdots\\
0 & 0 & \cdots & a_{dd}\\
 \end{pmatrix}\]
where all diagonal entries $a_{ii}$ are greater than zero.
Then we conclude that \[\det(\x_0,\ldots,\,\x_d)=\Big(\prod_{i=1}^d a_{ii}\Big)\det(\b_0,\ldots,\,\b_d).\] will always have the same sign, independently of the choice of the points $\x_i$. (This determinant cannot be $0$ as we require that the vectors $\x_0,\ldots,\,\x_d$ are linearly independent.)
\end{proof}
 
Now we will explain a second approach to prove that $\C(\R^{d},n)$ is a union of semialgebraic pieces. With the different concepts we have now, we can define when two partitions are combinatorially equivalent, and use this to construct the realization space of any partition $\P$ (made by all partitions that are combinatorially equivalent to $\P$).  
This will be useful in the discussion about the dimension of the spaces of convex $n$-partitions.

Given an $n$-partition $\P$, we want to describe all $n$-partitions that are combinatorially equivalent to $\P$. 
Two partitions $\P$ and $\P'$ have the same face poset if $\I(\P)=\I(\P')$. They have the same corresponding half-linear faces if the indices $I\in\I(\P)$ such that $F_I(\P)$ is half-linear are the same indices for which $F_I(\P')$ is half-linear.

\begin{definition}[Orientation of a partition]
  The \emph{orientation} of a partition $\P$ of $\R^d$ is given by the signs of the determinants $\det(\v_0,\ldots,\,\v_d)$ of all node frames of a node system $N$ of~$\P$.
\end{definition}

Orientations of partitions are closely related with orientations of cell complexes. If we consider the barycentric subdivision $S_N=\sd \P_N$ obtained by taking a point $\y_G$ in the relative interior of each face $G$ of $\P_N$, and with maximal simplices that are the spherical convex hull of sets $\y_{G_0},\ldots ,\y_{G_d}$ for each complete flag $G_0\subset \cdots\subset G_d$ in $\P_N$, then by Lemma \ref{signsofflags}, we can read an orientation of the simplicial complex $S_N$ from the orientation of~$\P$.

 Since orientations of oriented simplicial complexes are determined after fixing the orientation of one simplex, then it is enough to know the sign of one node basis to determine the sign of all other node bases of $\P_N$.
 In particular, if $\P$ is an essential partition, then any  node system on $\P$ will give rise to the same orientation. 
If $\P$ is non-essential, there are two possible orientations, depending on the choice of the nodes on the minimal face $F_{I({\bf 0})}$.

Orientations also keep track of which node frames are node basis and which are flats. 
 Two partitions $\P$ and $\P'$ with the same face poset and corresponding half-linear faces have the same orientation if there are  node systems $N$ and $N'$ on each of them, so that the sign  of the determinants of corresponding node basis are always the same and they have the same corresponding flats.  

\begin{definition}[Combinatorial type of a partition]\label{combinatorialtype}
The \emph{combinatorial type} of an $n$-partition $\P$ is given by the following information: the set $\I(\P)$ of labels of the face poset, the set of half-linear faces of~$\P$, and the orientation given by a node system of~$\P$.
\end{definition}
 
Orientations allow us to distinguish the combinatorial type of an essential partition and its 
reflection on a hyperplane. If a partition has some reflection symmetry, it implies that it is non-essential. 
Orientations also make sure that combinatorially equivalent partitions have the same $\pi$-angles, 
as defined next.
 
\begin{definition}[$\pi$-angles]\label{piangle}
Two $(d-1)$-faces $F_{ij}$ and $F_{ik}$ form a \emph{$\pi$-angle} if they belong to the same $(d-1)$-subface of a $d$-face $F_i$ of~$\P$ and their intersection is $(d-2)$-dimensional.  
This means that the dihedral angle between these two $(d-1)$-faces is equal to~$\pi$. 
\end{definition}
 
For the proof of Theorem \ref{descriptionsemialgebraic} we need a characterization for cone partitions from \cite{firlaZiegler}. 
A \emph{cone partition} (called simply a partition in that paper) of a cone $C$ is a collection of cones $C_1,\cdots,C_r$ contained in $C$ such that every point of $C$ is contained in one of the subcones $C_i$, where also the intersection of any two subcones $C_i \cap C_k$ is a face of both cones. 
 
\begin{theorem}[Firla--Ziegler {\cite[Theorem 4]{firlaZiegler}}]\label{fz}
A set of cones $C_1,\cdots,C_r$ of dimension $d+1$ contained in a bigger cone $C\subset\R^{d+1}$ form a cone partition of $C$ if and only if there is a generic vector $\g$ contained in exactly one of the cones $C_k$, and for any  $d$-face $F$ of a $(d+1)$-cone $C_i$ that is not contained in the boundary of $C$ there is a second cone $C_j$ with $C_i\cap C_j=F$ such that $F$ is a face of $C_j$.
\end{theorem}

\begin{theorem}\label{descriptionsemialgebraic}
Let $\P$ be a partition of $\R^d$ together with a node system $N$. Consider a list of vectors $\x_\v \in \R^{d+1}$ for every node $\v\in N$ that satisfy the following algebraic relationships and inequalities:
\begin{enumerate}[\rm(i)]\itemsep=0pt
\item $\|\x_\v\|=1$ for every node $\v$  in~$N$.
\item $\det(\x_{\v_0},\ldots,\x_{\v_d})> 0$, for every node basis $(\v_0,\ldots,\v_d)$ with $\det(\v_0,\ldots,\v_d)> 0$.
\item $\det(\x_{\v_0},\ldots,\x_{\v_d})= 0$, for every node flat $(\v_0,\ldots,\v_d)$.
\item $\e_0\cdot \x_\v=0$, for any node $\v\in N$ at infinity (i.e.\, in the boundary of $S^d_+$).
\item $\e_0\cdot \x_\v>0$, for any other node $\v\in N$, not at infinity.
\end{enumerate}
Assume also that there is a vector $\g\in \R^{d+1}$  that is generic (i.\,e.\ not contained in a hyperplane spanned by $d$ vectors $\x_{\v_i}$) that belongs to the interior of exactly one of the cones spanned by all vectors $\x_{\v}$ corresponding to the nodes $\v$ that belong to a $d$-face of~$\P_N$.

Then there is a partition $\P'$ that is combinatorially equivalent to $\P$ with a node system given by the points $x_\v$ for $\v \in N$.
\end{theorem}

\begin{proof}
We want to see first that we can construct a regular CW complex $\P_X$ by taking a face $G'$ for each face $G$ in $\P_N$, where $G'$ is the spherical convex hull of the points $\x_\v$ for all nodes $\v\in G$. Then we will obtain the partition $\P'$ out of the complex $\P_X$.

Consider the barycentric subdivision $\sd \P_N$ of the complex $\P_N$ obtained by taking  points $\y_G$ in the relative interior of each face $G$ of $\P_N$. The maximal simplices of $\sd \P_N$ correspond to complete flags $G_0\subset \cdots\subset G_d$ in $\P_N$ and have $\y_{G_0},\ldots ,\y_{G_d}$ as vertices. Then take a point $\y'_{G}$ in the relative interior of each spherical polyhedral set $G'$ in $\P_X$. 
We want to see that if we construct the family $S_X$ of simplicial cones over the sets $\y'_{G_0},\ldots,\y'_{G_d}$ for each complete flag $G_0\subset \cdots\subset G_d$ in $\P_N$, then we obtain a cone partition of the upper  halfspace of $\R^{d+1}$ (with first coordinate $x_0\geq 0$), by making use of Theorem \ref{fz}.

\begin{lemma}\label{samectypeofcones}
Let $G$ be a $d$-face of $\P_N$. Then the algebraic relationships and inequalities for node frames (of type {\rm (ii)} and {\rm (iii)}) imply that $G'$ is combinatorially equivalent to $G$ as a polyhedral cone.
\end{lemma}

\begin{proof}  
The relationships of type {\rm (iii)} coming from flats tell us that the points $\x_\v$ corresponding to nodes $\v$ on the same $d$-subface of $G$ are all on the same hyperplane and the inequalities of type {\rm (ii)} for node bases tell us that this hyperplane defines a facet of $G'$. Moreover, for each node $\v$, the set of facets on $G$ where it belongs must be similar to the set of facets of~$G'$ where the point $\x_\v$ is contained. 

We can tell which nodes are vertices of $G$ from the set of facets where each node belong. Vertices are in the maximal sets under inclusion, because if a node $\v$ is not a vertex, the set $A_\v$ of facets of $G$ containing $\v$ is determined by the subface of $G$ that contains it, and this is a subset of the set $A_{\v'}$ of facets of $G$ containing a vertex $\v'$ of that subface. Therefore $G$ and $G'$ have the same vertex-facet incidences, and this imply that they are combinatorially equivalent (this is a direct consequence of the analogous result for convex polytopes, see \cite[Lect.~2]{Z}).
\end{proof}  

The cone over $G'$ is subdivided by all cones of the form $\cone(\y'_{G_0},\ldots,\y'_{G_d})$
associated to complete flags $G_0\subset \cdots\subset G_d$ on with $G=G_d$.
By assumption, there is a generic vector $\g$ contained in exactly one of the cones spanned by the  vectors $\x_{\v}$ for all nodes $\v$ that belong to a $d$-face $G$ of the complex $\P_N$. This is precisely the cone over the spherical polyhedron~$G'$.%

Since the vector $\g$ is generic, it will belong to the interior of exactly one of the subcones $\cone(\y'_{G_0},\ldots,\y'_{G_d})$ corresponding to a complete flag with $G=G_d$. By a similar argument, if $G_d\neq G$, it is not possible that $\g$ belong to any other cone corresponding to a flag ending in $G_d$ and $\g$ is in the interior of a unique cone from $S_X$. We conclude that the vector $\g$ belong to the interior of exactly one of the subcones $\cone(\y'_{G_0},\ldots,\y'_{G_d})$ corresponding to a complete flag with $G=G_d$, and therefore $\g$ is in the interior of a unique cone in $S_X$.  

Now we want to see that for any  $d$-face $F$ of a $(d+1)$-cone $C_i$ in $S_X$ that is not contained in the boundary of the upper halfplane in $\R^{d+1}$ there is a second cone $C_j$ with $C_i\cap C_j=F$ such that $F$ is a face of $C_j$.
Notice that the cones spanned by $\y_{G_0},\ldots,\y_{G_d}$ form a simplicial cone partition $S_N$ of the upper halfspace of $\R^{d+1}$, since they arise from a barycentric subdivision of~$\P_N$.%

\begin{lemma}\label{samesignsofdets}
For any complete flag $G_0\subset \ldots\subset G_d$,
the determinant $\det(\y'_{G_0},\ldots,\y'_{G_d})$ has the same sign as the determinant  
$\det(\y_{G_0},\ldots,\y_{G_d})$, 
\end{lemma}

\begin{proof}
By Lemma \ref{signsofflags} we know that the sign of the determinant $\det(\y_{G_0},\ldots,\y_{G_d})$ is the same than the sign of $\det(\v_0,\ldots,\,\v_d)$ for any  node basis $(\v_0,\ldots,\,\v_d)$ in $N$ with $\v_i\in G_i$.  

By the algebraic conditions on the $\x_\v$, this sign is also the same as that of the determinant $\det(\x_{v_0},\ldots,\x_{v_d})$ for any node basis $(\v_0,\ldots,\,\v_d)$ in $N$ with $\v_i\in G_i$. We want to see that the determinant $\det(\y'_{G_0},\ldots,\y'_{G_d})$ also has the same sign.

The fact that $\y'_{G}\in \relint G'$ can be expressed by a linear combination 
\[
        \y'_{G}=\sum_{\v\in N\cap G}\alpha_\v \x_\v,
\] 
where all $\alpha_\v>0$.
Since determinants are multilinear, we can expand as follows.
\[
        \det(\y'_{G_0},\ldots,\y'_{G_d})\ =
        \sum_{(\v_0,\ldots,\,\v_d)}   \Big(\prod_{i=0}^d \alpha_{\v_i}\Big)\det(\x_{v_0},\ldots,\x_{v_d}),
\]
where the sum goes over all lists $(\v_0,\ldots,\,\v_d)$ such that $\v_i\in G_i$, namely the node systems on the flag $G_0\subset \cdots\subset G_d$. We can see that all summands on the right have the same sign as $\det(\y_{G_0},\ldots,\y_{G_d})$ or are zero. 
\end{proof}

Lemma \ref{samesignsofdets} imply that two adjacent cones in $S_X$ don't overlap on their interiors, since the corresponding cones in $S_N$ don't overlap. All $d$-faces of $S_X$ corresponding to faces on the boundary of $S_N$ are also in the boundary of the upper halfspace (due to relationships of type {\rm (iv)}) while a $d$-face of a cone $C_i\in S_X$  corresponding to an interior $d$-face of $S_N$ are always interior (due to the inequalities of type {\rm (v)}), and by the lemma we can find that there is a second cone in $S_X$ such that its intersection with $C_i$ is the corresponding $d$-face, by looking at the cone with analogous property in $S_N$.

Now we are in conditions to use Theorem \ref{fz} to conclude that the cones in $S_X$ don't overlap and make a cone partition of the upper hemisphere. 

Each of the faces $G'$ of the $\P_X$ can be obtained as the intersection of $S^d$ with the union of the cones over sets $\y'_{G_0},\ldots,\y'_{G_k}$ where $G_0\subset \ldots\subset G_k=G$ are partial flags on $\P_N$. We can see that $\P_X$ is a CW complex since the relative interiors of its faces are pairwise disjoint and that the boundary of each face $G'$ is covered by the faces of $\P_X$ contained in $G'$, since by construction we have inclusion between faces $G'_1\subset G'_2$ if and only if the corresponding faces in $\P_N$ satisfy that $G_1\subset G_2$. 
Also the resulting complex $\P_X$ will have the same face poset as $\P_N$. Half-linear faces $F$ of~$\P$ can be obtained as union of faces of $\P_N$, and the union $F'$ of the corresponding faces of $\P_X$ have to be in a linear subspace of the right dimension, due to equations of type {\rm (iii)} that tell that points $\x_\v$ for $\v\in F$ have to be coplanar for all facets of all regions of~$\P$ containing $F$, and besides, $F'$ have in the boundary the same faces at infinity as $F$ (due to equations of type {\rm (iv)}), so $F'$ will be a half-linear face for a new partition $\P'$ that have as faces in its spherical representation the same faces as $\P_X$, but gluing together those faces corresponding to the same half-linear face of $\P$. 

The fact that $\P'$ is a partition of $\R^d$ is a consequence that $S_X$ is a cone partition of the upper halfspace. By Lemma \ref{samesignsofdets} we can find that the $\P$ and $\P'$ have the same orientations, and we conclude that they are combinatorially equivalent as we wanted. 
\end{proof}

The condition of the existence of a vector $\g$ in the interior of only one of the $d$-faces of $\P_X$ is important and cannot be omitted. To see this, consider a  5-partition of $\R^2$ as in the left of Figure \ref{5partitionproblem}, and the choice of points $x_{\v_i}$ depicted on the right.  For simplicity we called the vertices $\v_i$ and all nodes are vertices since the partition is pointed. In that example, all conditions from Theorem \ref{descriptionsemialgebraic} are satisfied, except the existence of the point $\g$. In this case we get that the expected spherical regions form a double covering of the upper hemisphere.
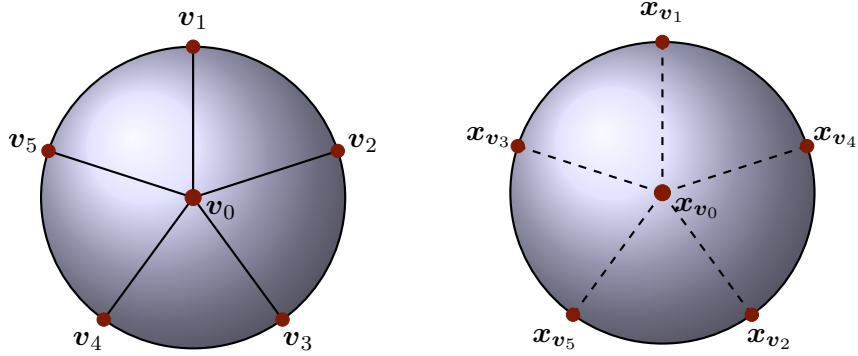
\begin{figure}[htb]\centering

  \begin{tikzpicture} 
    \shade [ball color=blue!15]
    (0,0) circle (2);

    \draw [thick] (0, 0) circle (2);


\def \n {5}
    
\foreach \s in {1,...,\n} {
  \pgfmathsetmacro {\angle}{90-(360/\n * (\s - 1))}
    \draw [thick] (0, 0) -- (\angle:2);  
    \draw [fill,tinto] (\angle: 2) circle (2.5pt);
    \node at (\angle:2.35) {$\v_{\s}$};
}

    \draw [fill,tinto] (0, 0) circle (3pt);
    \node at (-24:0.4) {$\v_{0}$};

  \end{tikzpicture}
\qquad 
  \begin{tikzpicture} 
    \shade [ball color=blue!15]
    (0,0) circle (2);

    \draw [thick] (0, 0) circle (2);


\def \n {5}
    
\foreach \s in {1,...,\n} {
  \pgfmathsetmacro {\angle}{90-2*(360/\n * (\s - 1))}
    \draw [thick, dashed] (0, 0) -- (\angle:2);  
    \draw [fill,tinto] (\angle: 2) circle (2.5pt);
    \node at (\angle:2.4) {$\x_{\v_{\s}}$};
    
}

    \draw [fill,tinto] (0, 0) circle (3pt);
    \node at (-24:0.5) {$\x_{\v_{0}}$};
  \end{tikzpicture}

    \caption{Nodes of a $5$-partition together with points $\x_\v$ that satisfy all algebraic relationships and inequalities in Theorem \ref{descriptionsemialgebraic} but don't make a new $5$-partition.}
\label{5partitionproblem} 
\end{figure}
   
\begin{proposition}\label{PNspace}
Let $\P$ be an $n$-partition of $\R^d$. The space of pairs $(\P',N')$ of partitions $\P'$ combinatorially equivalent to together with a node system $N'$ on $\P'$ is a semialgebraic set.%
\end{proposition}

\begin{proof}
Theorem \ref{descriptionsemialgebraic} gives an algebraic description by equations and inequalities of a set that parameterizes all these pairs, under the condition of the existence of the point $\g$. Notice that if a partition $\P'$ is combinatorially equivalent to $\P$, then any node system give rise to an equivalent system of equations and inequalities, and therefore it satisfies the system given by~$\P$. The condition about the point $\g$ can be also given as a system of algebraic conditions after introducing new slack variables for $\g$. We recall that unions and intersections of semialgebraic sets are semialgebraic. Then by Theorem \ref{tarskiseidenberg} we find that the set we are interested in is semialgebraic.
\end{proof} 

\begin{definition}[Realization spaces]
  The \emph{realization space} of an $n$-partition $\P$ is the subspace of $\C(\R^{d},n)$ of all partitions $\P'$ with the same combinatorial type as $\P$. It is denoted as $\C_{\P}(\R^d,n)$. 
\end{definition}

\begin{theorem}\label{semialgebraic}
Let $\P$ be an $n$-partition of $\R^d$. Then the realization space $\C_{\P}(\R^d,n)$ is a semialgebraic set.
\end{theorem} 

\begin{proof} 
Proposition \ref{PNspace} shows that for pointed partitions $\P$ the  space $\C_{\P}(\R^d,n)$ is semialgebraic, since all vertices are nodes, and there is a unique node system on each partition in the realization space.
In general, the realization space of~$\P$ can be obtained as the image of the space of pairs $(\P',N')$ described in Proposition \ref{PNspace} to the space $\R^{h(d+1)}$ describing by the equations of the $h$ hyperplanes that define $(d-1)$-faces of the partition, where each partition corresponds a unique point.  We make use of an equivalent formulation of Theorem \ref{tarskiseidenberg} that claims that the image under a polynomial mapping $f:\R^m\rightarrow \R^{m'}$  of a semialgebraic set is semialgebraic (see \cite[Proposition 2.83]{basu}). 
\end{proof}

This result gives us an alternative proof of the fact that spaces of $n$-partitions $\C(\R^{d},n)$ are union of semialgebraic pieces, the union of all realization spaces of $n$-partitions of~$\R^d$ is equal to $\C(\R^d,n)$.


\section{Examples}\label{examples}

We will analyze here the spaces of $n$-partitions for small values of $n$ and $d$. 
For $n=1$ the space of partitions $\C(\R^{d},1)$ will simply consists of one point. 
A more interesting but still easy case is $n=2$.

\begin{proposition}\label{sphereford2}
The space $\C(\R^{d},\le\!2)$ is homeomorphic to the sphere $S^d$. The space of partitions $\C(\R^{d},2)$ is homotopy equivalent to $S^{d-1}$ and is obtained from $\C(\R^{d},\le\!2)$ by removing two points. 
\end{proposition}

\begin{proof}
To parameterize our space of $2$-partitions for fixed $d$ we only need to choose the coordinates $\c_{1,2}$, that describe the normal to the hyperplane $H_{ij}$ by a point in $S^n$. Two special cases have to be taken into account that characterize the cases when the combinatorial type of the $2$-partition is not the generic one. These are precisely when $\c_{ij}=\pm (1,0,\ldots,0)$. In those cases, there is no hyperplane in $\R^d$, representing the partitions with only one (labeled) non-empty region. These extreme partitions can be obtained as a limit of proper $2$-partitions, and $S^d$ will handle the topological structure of $\C(\R^{d},\le\!2)$ in the right way.
\end{proof}

For $n\ge 3,$ things begin to be more complicated, even in the case of $d=1$.

\begin{proposition}\label{nfactorial}
 The space $\C(\R^{1},\le\!n)$ is homeomorphic to a CW complex with $n$ vertices and  $k!\binom{n}{k}$ simplicial $(k-1)$-cells for $0\le k\le n$. It is made out of $n!$ simplices of dimension $(n-1)$ glued appropriately on the boundaries. The space $\C(\R^{1},n)$ is homeomorphic to $n!$ open $(n-1)$-balls.  
\end{proposition}

\begin{proof}
For a combinatorial type with $k$ non-empty regions, its realization space is contractible and can be realized as a $(k-1)$-simplex. To do this, take an order preserving homeomorphism from $\R$ to the open interval $(0,1)$. Then the coordinates of the $k-1$ interior vertices (hyperplanes!) $\v_{i,j}=F_{i,j}\in\R$ need to be in a prescribed order, and via the homeomorphism we can map any partition to a point inside a $(k-1)$ -simplex contained in the unit cube $(0,1)^{k-1}$. 

For example, if the $n$-partition have the region $i$ at the left of region $i+1$ for all $i<n$ (and no empty region) then we only need to specify the coordinates of the vertices $\v_{i,i+1}$ such that $\v_{1,2}\le\ldots\le\v_{n-1,n}$. Mapping these $n-1$ values to the unit cube $(0,1)^{n-1}$ via the homeomorphism, we identify the realization space of this particular $n$-partition with the interior of an $(n-1)$-simplex.

The boundary of each of those simplices will represent the case when some of the points coincide, and can be naturally identified with the realization spaces of other combinatorial types with some extra empty regions. In this way we give to $\C(\R^{1},\le\!n)$ the structure of a regular cell complex (start with $n$ vertices corresponding to the realization spaces of partitions with only one non-empty region, and then for higher dimensions, identify the boundary with a subspace of the union of the cells of smaller dimension).  

There will be $n!$ combinatorial types without empty regions. The space $\C(\R^{1},n)$ of proper partitions of $\R$ is the union of the interior of all those simplices. All other combinatorial types can be obtained in the limit (in the boundary) of those proper combinatorial types and therefore $\C(\R^{1},\le\!n)$ will have $n!$ top-dimensional simplicial $(n-1)$-cells, and $\binom{n}{k}k!$ cells of dimension $k-1$.  
\end{proof}

\begin{example}
The space $\C(\R^{1},\le\!3)$ is homeomorphic to a two dimensional space made out topologically by gluing six simplices along the boundaries in a special way, since there are two different edges joining each pair of vertices. The vertices represent the partitions with one non-empty region, and the edges represent the partitions with two non-empty regions.
\begin{figure}[htb]\centering
  
  \begin{tikzpicture}
\foreach \x/\y/\pa/\pb/\pc in {0/0/1/2/3, 3/0/1/3/2, 6/0/2/1/3, 0/-3/2/3/1, 3/-3/3/1/2, 6/-3/3/2/1}{ 
    \node[]  (na) at (\x,\y)  {};           
    \node[]  (nb) at ($(na)+(1.5,0)$)  {}; 
    \node[]  (nc) at ($(nb)+(0,1.5)$) {};  

    \fill[gray]  ($(na)$)--($(nb)$)--($(nc)$)--cycle;
    \draw [thick]  ($(na)$)--($(nb)$) node[midway,sloped,below] {\small $\pb\pc$}--($(nc)$) node[midway,sloped,below] {\small $\pa\pb$}--($(na)$) node[midway,sloped,above] {\small $\pa\pc$};

    \fill (na) circle (2pt);
    \fill (nb) circle (2pt);
    \fill (nc) circle (2pt);

 \node    at ($(\x,\y)+(-0.3,-0.2)$) {\small $\pc$};
 \node    at ($(\x,\y)+(1.7,-0.2)$) {\small $\pb$};
 \node    at ($(\x,\y)+(1.7,1.8)$) {\small $\pa$};
 \node    at ($(\x,\y)+(1,0.4)$) {\small $\pa\pb\pc$};
  
}
\end{tikzpicture}

\caption{Simplices to build a cell complex homeomorphic to $\C(\R^{1},3)$.}\label{3partitionsofr}
\end{figure}
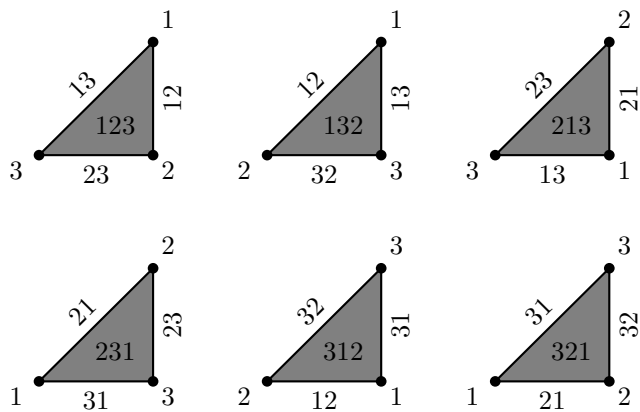

In Figure \ref{3partitionsofr} we can see the six simplices of this CW complex with the corresponding labels on the different cells. These simplices have to be glued along the edges corresponding  to the same partitions, in such way that the corresponding vertices coincide. Each edge appears in three of the simplices.
\end{example}

As a further example, in \cite[Section 7.3]{thesis} we give a complete description of
the space $\C(\R^{2},3)$ as well as a cell complex model for its closure $\C(\R^{2},\le\!3)$.
 
 
\section{Further results}\label{furtherresults} 

Inside the space $\C(\R^d,n)$ there are other spaces that catch our attention, such as the subspace $\C_{\reg}(\R^d,n)$ of regular partitions, which can be obtained by projecting the facets of a convex polyhedron one dimension higher. Regular partitions appear in different contexts and are much better understood than general partitions, since they are easier to generate and parameterize. We would like to know how the space of regular partitions is embedded in the space of all convex $n$-partitions. 

We find that there is a big difference between the case $d=2$ and the case when $d\geq 3$.  For $d=2$  and large $n$, the subspace $\C_{\reg}(\R^{2},n)$ of regular $n$-partitions has much smaller dimension than $\C(\R^{2},n)$, as it can be seen from the following results (see \cite{thesis}). 

\begin{theorem}\label{mdim2}
For $n\ge3$ the space $\C(\R^{2},n)$ of partitions of $\R^2$ into $n$ convex pieces has dimension 
$\dim \C(\R^{2},n)=4n-7$. 
The partitions whose realization spaces attain the top dimension are simple with exactly three unbounded regions.
\end{theorem}

\begin{theorem}\label{regulardimension}
For $d\geq 2$ and $n\geq 2$, the space $\C_{\reg}(\R^{d},n)$ of regular partitions is a semialgebraic set of dimension
\[ 
\dim \C_{\reg}(\R^{d},n)=(d+1)(n-1)-1.
\]
In particular, the space of regular $n$-partitions of the plane has dimension $\dim \C_{\reg}(\R^{2},n)=3n-4$. 
\end{theorem}  

For $d\geq 3$, however, a theorem by Whiteley \cite{Whiteley3}, 
generalized by Rybnikov \cite{rybnikovStresses}, shows that simple $n$-partitions are regular.
\begin{conjecture} 
$\dim \C(\R^{3},n)=\dim \C_{\reg}(\R^{3},n)=4n-5$. 
\end{conjecture} 
However, $\C_{\reg}(\R^{3},n)$ is not a dense subset in $\C(\R^{3},n)$ for $n>3$, and there are also non-simple combinatorial types whose realization spaces have the same dimension as $\C_{\reg}(\R^{3},n)$, where partitions are generically non-regular. 

In general, realization spaces of partitions of a given combinatorial type are expected to be  complicated objects. We relate this to the work by Richter-Gebert \cite{Rich4} on realization spaces of polytopes, where the main result is the Universality Theorem, showing that realization spaces of $d$-dimensional polytopes for $d\ge 4$ can be ``as complicated as possible'' as semialgebraic sets. A similar result is established for realization spaces of regular partitions \cite[Theorem 5.17]{thesis}.

\begin{theorem}\label{universal}
    For any primary basic semialgebraic set $X$ and $d\ge3$, there is an $n$-partition $\P$ of $\R^d$ such that the set of regular partitions combinatorially equivalent to $\P$, up to affine equivalence, form a semialgebraic set stably equivalent to $X$. 
\end{theorem}

\end{document}